\documentclass[11pt,reqno,twoside,a4paper]{article}
\usepackage{mysty}
\usepackage{wrapfig}
\usepackage{fullpage}
\title{Low Complexity Regularized Phase Retrieval}

\author{Jean-Jacques Godeme\thanks{Normandie Univ, ENSICAEN, CNRS, GREYC, France. e-mail: \texttt{jean-jacques.godeme@unicaen.fr, Jalal.Fadili@ensicaen.fr}.} \and Jalal Fadili\footnotemark[1]}  

\date{}

\begin{document}

\maketitle
\begin{flushleft}\end{flushleft}

\begin{abstract}
In this paper, we study the phase retrieval problem in the situation where the vector to be recovered has an a priori structure that can encoded into a regularization term. This regularizer is intended to promote solutions conforming to some notion of simplicity or low complexity. We investigate both noiseless recovery and stability to noise and provide a very general and unified analysis framework that goes far beyond the sparse phase retrieval mostly considered in the literature. In the noiseless case we provide sufficient conditions under which exact recovery, up to global sign change, is possible. For Gaussian measurement maps, we also provide a sample complexity bound for exact recovery. This bound depends on the Gaussian width of the descent cone at the sought-after vector which is a geometric measure of the complexity of the latter. In the noisy case, we consider both the constrained (Mozorov) and penalized (Tikhonov) formulations. We provide sufficient conditions for stable recovery and prove linear convergence for sufficiently small noise. For Gaussian measurements, we again give a sample complexity bound for linear convergence to hold with high probability. This bound scales linearly in the intrinsic dimension of the sought-after vector but only logarithmically in the ambient dimension. 
\end{abstract}



\section{Introduction}\label{sec:intro}
\subsection{The phase retrieval problem}
The phase retrieval problem arises in many applications including X-ray crystallography, diffraction imaging, and light scattering, to name just a few; see \cite{shechtman_phase_2014,JaganathanReview16,luke_phase_2017} and references therein. Phase retrieval is an active research area and we refer to \cite{shechtman_phase_2014,JaganathanReview16,luke_phase_2017,fannjiang_numerics_2020,vaswani_non-convex_2020} for recent reviews of the current state of the art.

Our focus in this paper will be on the case of real signals. In this case, phase retrieval consists of recovering $\avx\in\bbR^n$ from phaseless, possibly noisy, real or complex measurements. In the noiseless case, one measures the squared modulus of the inner product between $\avx$ and $m$ sensing/measurement vectors $(a_r)_{r \in \bbrac{m}}$, ${\bbrac{m}} \eqdef \{1,\ldots,m\}$, and the goal is to recover $\avx$ from the intensities $(\valabs{\pscal{a_r,\avx}})_{r\in \bbrac{m}}$. In real applications, however, the intensity measurements are not perfectly acquired. For instance, let us consider light scattering for precision in optics~\cite{amra_instantaneous_2018,buet2022immediate,buet2023instantaneous} which is our motivating application. Here, the goal is to describe the roughness of a polished surface. The latter is illuminated with a laser source, and the diffusion is measured by moving a detector. Then the power spectral density of the surface topography can be directly measured. However, during the acquisition process, different types of noise can corrupt the measurements such as photon noise, thermal noise, Johnson noise, \etc. To account for noise, we thus consider a generic additive noise model in which case the noisy phase retrieval problem reads:
\begin{equation}\tag{GeneralPR}\label{eq:GeneralPR}
\begin{cases}
\text{Recover $\avx\in \bbR^n$ from the measurements $y \in \bbR^m$} \\
y[r]=\valabs{\pscal{a_r,\avx}}^{2}+\epsi[r], \quad r \in \bbrac{m}, 
\end{cases}
\end{equation}
where $[r]$ is the $r$-th entry of the corresponding vector, and $\epsi\in\bbR^m$ is the noise vector. The measurement model \eqref{eq:GeneralPR} is quite standard and is similar for instance to \cite{candes_phaselift_2013,demanet_stable_2013,chen_solving_2017}.

Since $\avx$ is real-valued, the best one can hope for is to ensure that $\avx$ is uniquely determined from its intensities up to a global sign. In fact, phase retrieval is a severely ill-posed inverse problem in general, and even for $\epsilon=0$, checking whether a solution to \eqref{eq:GeneralPR} exists or not is known to be NP complete \cite{Sahinoglou91}. The situation is even more challenging in presence of noise and stability can be only ensured to sufficiently small noise. Thus, one of the major challenges is to design optimization problems and efficient recovery algorithms and find conditions on $m$, $(a_r)_{r \in {\bbrac{m}}}$ and $\epsilon$ which guarantee exact (up to a global sign change) and stable recovery to small enough noise.

In order to reach the land of well-posedness without unreasonably increasing the number of measurements (\ie oversampling), it appears natural to restrict the inversion process to a well-chosen low dimensional subset of $\bbR^n$ containing the plausible solutions including $\calX \eqdef \{\pm\avx\}$. It is then natural to leverage this low dimensional structure which will hopefully allow to minimize the number of measurements needed for recovery, and this is the most important as the measurement process might be expensive or can destroy the sample at hand. A standard way to implement this idea consists in adopting a variational framework where the sought-after solutions are those where a prior penalty/regularization function is the smallest. This approach is in line with variational regularization theory pioneered by Tikhonov \cite{tikhonov1978solutions}. Put formally, this amounts to solving the following optimization problem
\begin{equation}\tag{$\scrP_{y,\lambda}$}\label{eq:eqminregulPRpen}
\inf_{x\in\bbR^n} \ens{F_{y,\lambda}(x) \eqdef \normm{y - |Ax|^2}^2 + \lambda R(x)} , 
\end{equation}
where $A=\transp{[a_1,\ldots,a_m]}$ and $R: \bbR^n \to \bbR \cup \{\pinf\}$ is a proper closed convex function which is intended to promote objects similar or close to $\avx$. $\lambda > 0$ is the regularization parameter which balances the trade-off between fidelity and regularization. It is immediate that the data fidelity function $x \in  \bbR^n\mapsto\normm{y - |Ax|^2}^2$ is non-convex due to the quadratic measurements (though weakly convex). Besides, it is $C^2(\bbR^n)$ but its gradient is not Lipschitz continuous. In this setting, we can associate to this data fidelity term following entropy or kernel function
\begin{equation}\label{eq:entropy}
\psi(x)=\frac{1}{4}\normm{x}^4+\frac{1}{2}\normm{x}^2. 
\end{equation}
It is known that $\normm{y - |A \cdot|^2}^2$ is so-called smooth relative to $\psi$ \cite{bolte_first_2017}, meaning that it verifies an appropriate descent lemma wrt to $\psi$; see \eqref{eq:relsmooth}. Therefore  $F_{y,\lambda}(x)$ is amenable to the efficient Bregman Proximal Gradient algorithm described and discussed in Section~\ref{sec:numexp}.

\medskip

It is well known in the inverse problem literature, see e.g. \cite{scherzer2009}, that the value of $\lambda$ should typically be an increasing function of $\normm{\epsi}$. In the special case where there is no noise, \ie $\epsi = 0$, the fidelity to data should be perfect, which corresponds to considering the limit\footnote{This will be studied rigorously in Section~\ref{sec:robustpenalized}.} of \eqref{eq:eqminregulPRpen} as $\lambda \to 0^+$. This limit turns out to be the noiseless version for exact (up to a global sign) recovery,
\begin{equation}\tag{$\scrP_{\avy,0}$}\label{eq:eqminregulPRexact}
\inf_{x\in\bbR^n} R(x) \qsubjq |Ax| = \sqrt{\avy} \qwhereq \avy \eqdef |A\avx|^2 .
\end{equation}
Denoting $\overbar{\calF} \eqdef \enscond{w \in \bbR^m}{|w|= \sqrt{\avy}}$, which is a non-empty finite set of cardinality $2^m$ (vertices of a hyper-rectangle), \eqref{eq:eqminregulPRexact} can be equivalently written as
\begin{equation*}
	\inf_{x\in\bbR^n} R(x) \qsubjq A x \in \overbar{\calF}.  
\end{equation*}
\subsection{Prior work}
Regularized phase retrieval  is an active area of research. Our review of this problem is not exhaustive and readers interested in a comprehensive and extended overview should refer to the following references \cite{shechtman_phase_2014,JaganathanReview16,vaswani_non-convex_2020,luke_phase_2017,fannjiang_numerics_2020}. 

\paragraph{Sparse phase retrieval} When the signal of interest is $s-$sparse w.r.t some basis and the goal is to recover the signal from fewer measurements $m\ll n$, this problem is referred as ``compressive or sparse phase retrieval''.  From a theoretical perspective, generic sensing vectors $(a_r)_{r\in \bbrac{m}}$  are injective (up to a global sign change) in the class of real $s-$sparse signals as soon as the number of measurements satisfies $m\geq 2s-1$ \cite{wang_phase_2014}. We recall that the natural information theoretical lower-bound is $m\gtrsim s\log(n)$ for solving the problem using any approach. Whereas for Gaussian sensing vectors, \cite{ohlsson_conditions_2014} show that $m \gtrsim s\log(en/s)$ separate signals well. 
In \cite{Voroninski14}, the authors introduced a notion of strong Restricted Isometry Property (s-RIP) which holds for the class of Gaussian sensing vectors and they showed that solving \eqref{eq:eqminregulPRexact} when $R$ is the $\ell_0-$norm is equivalent to solving the same problem replacing $\ell_0$ with the $\ell_1$ norm for sensing vectors satisfying the s-RIP. For Gaussian sensing vectors, the latter holds for $m \gtrsim s \log(en/s)$. Stable sparse phase retrieval under the s-RIP was studied in \cite{gao_stable_2016}. Other works in the same vein include \cite{tarokh_new_2013,yang_misspecified_2019,gao_stable_2016,wang_phase_2014,schniter_compressive_2015,jaganathan_recovery_2012,jaganathan_sparse_2013,jaganathan_sparse_2017,bandeira_near-optimal_2013}.

We can categorize the methods to solve the sparse phase retrieval problem into three groups. The first considers convex relaxation, the second tackles directly the non-convex problem and the third manually designs the sensing vectors. 

Lifting methods such as the PhaseLift can be used to convexify the constraint in \eqref{eq:eqminregulPRpen} while sparsity on the lifted rank-one matrix is now to be promoted entry-wise or row-wise. This regularization entails that the rank-one matrix to be recovered is $s^2$ sparse and thus, as expected, the sample complexity for exact recovery from Gaussian measurements is $m \gtrsim s^2\log(n)$ \cite{li_sparse_2013}. However, this problem does not scale with $s$ and it is not possible to achieve the natural theoretical lower-bound using this approach \cite{li_sparse_2013,oymak_simultaneously_2015}. Another approach in this setting is PhaseMax \cite{hand_compressed_2016} which consists in relaxing the non-convex constraint set in \eqref{eq:eqminregulPRexact} from equality to inequality (\ie from the sphere constraint to the ball one), and then to solve the resulting linear program. This method achieves the optimal sample complexity  $m \gtrsim s\log(n/s)$ for Gaussian sensing vectors. However, it requires an anchor or initialization that is sufficiently correlated with the true signal which requires $m \gtrsim s^2\log(n)$ to be successful. In \cite{mcrae_optimal_2023}, the authors use a convex relaxation and propose an atomic norm that favours low-rank and sparse matrices. They achieve nearly optimal sample complexity \ie bound $m \gtrsim s\log(en/s)$. Regarding the stability of the reconstruction against additive noise, the same authors showed that the sparse with low-rank atomic norm regularization achieves a reconstruction error bound of $O(\sigma\sqrt{(s\log(en/s))/m})$ where $\sigma$ is the noise standard deviation.    

Concerning methods that study directly the sparse phase retrieval problem, it has been shown that $m \gtrsim s^2\log(n)$ are sufficient to provably recover the original vector (up to global sign/phase change)  \cite{cai_optimal_2016,netrapalli_phase_2015,wang_sparta_2017,yuan_phase_2019,jagatap_fast_2017}.
The authors in \cite{cai_optimal_2016} proposed a method to find a good initialization of the problem which requires that $m \gtrsim s^2\log(n)$. The authors in \cite{netrapalli_phase_2015}  proposed an alternating minimization strategy to reconstruct the signal. Sparta \cite{wang_sparta_2017} uses an amplitude-based instead of an intensity-based measurement which is clearly non-smooth and \cite{yuan_phase_2019} proposed a sparse truncated version of the classical Wirtinger flow \cite{Candes_WF_2015}. The authors in \cite{jagatap_fast_2017} proposed the Copram which combines Alternating minimization and the Cosamp \cite{needell_cosamp_2009}, and they showed that reconstruction is possible with $m \gtrsim s^2\log(n)$ measurements. In the general case of block sparsity or group Lasso, they showed that exact recovery (up to a global sign change) is possible with $m\gtrsim\frac{s^2}{B}\log(n)$ where $B$ is the size number of the blocks slightly improving the bound on the number of measurements. 
As far as robust recovery is concerned, a thresholded Wirtinger flow for noisy sparse phase retrieval was proposed and analyzed in \cite{cai_optimal_2016} (see also \cite{yang_misspecified_2019} for an extension to for misspecified phase retrieval). For $A$ standard Gaussian and random noise with zero-mean independent centred sub-exponential entries, it was shown that when properly initialized, this procedure with $m \gtrsim s^2\log(n)$ achieves a reconstruction error of $O(\sigma/\normm{\avx}\sqrt{(s\log(n))/m})$, where $\sigma$ characterizes the noise level. This rate is essentially optimal as proved in \cite{lecue_minimax_2015}.  

If one has complete control over how the measurements vectors are designed, then near-optimal sample complexity bounds with practical algorithms can be obtained \cite{bahmani_efficient_2015,pedarsani_phasecode_2017}. This is however of limited interest to us as we are primarily concerned with generic measurement vectors.

\paragraph{General regularized phase retrieval} As reviewed above, most existing work focuses on the recovery of sparse signals from phaseless measurements. On the other hand, real signals and images involve much richer structure and complexity such as being piecewise smooth. In this case, a wise choice of the regularizer would be the popular Total Variation (TV) seminorm, or group-sparsity in some frame. This scope is quite recent for the phase retrieval. For the TV phase retrieval, we refer to \cite{beinert_total_2022,beinert_total_2023}.  In \cite{beinert_total_2022}, the authors combined the standard Fienup's Hybrid input-output \cite{fienup_phase_1982} method that is well-known to be the Douglas-Rachford \cite{bauschke_phase_2002} with TV regularization based on a primal-dual method. This was applied to optical diffraction tomography and the sensing vectors are the Non-Uniform Fourier Transform. In \cite{beinert_total_2023}, they extend the scope to moving objects. See also \cite{pauwels_fienup_2018} for an algorithmic framework based on Fienup methods with general semialgebraic regularizers. None of those works proved recovery and stability results.

In the general setting, we have to cite the work of \cite{lecue_minimax_2015}, where the authors consider the reconstruction of a real vector living in a general model subset $\Omega \subset \bbR^n$ from sub-Gaussian measurements. They showed that empirical risk minimization (ERM) of $\Omega$ to solve the noisy phase retrieval produces a signal close enough to the true signal (up to a global sign change) and this error depends on the Gaussian width of $\Omega$ and the signal-to-noise ratio of the problem. Phase retrieval with general regularization is studied in \cite{soltanolkotabi_structured_2019}, where the authors showed that the main problem for achieving the optimal sample complexity is the initialization step. However, it is still an open question to cook up a general good recipe to find an anchor or initialization that is close enough or sufficiently correlated with the true vector beyond the sparse case, and with a reasonable bound on the number of measurements, \ie that does not scale as the square of the intrinsic dimension of the vector to recover.

\subsection{Contributions and relation to prior work}
In this paper, we start by providing sufficient conditions under which the set of solutions to \eqref{eq:eqminregulPRexact} is non-empty. Then, we deliver a unified analysis showing that the recovery of $\avx$ up to a global sign is exact when we solve \eqref{eq:eqminregulPRexact} under two geometric (deterministic) conditions on $R$, the descent cone of $R$ and the deterministic measurements $A$. It turns out that for standard Gaussian measurements and the class of regularizer that we consider, these conditions are satisfied with high probability under a sufficiently large sample complexity. As a consequence, when the number of measurements is large enough the recovery of $\avx$ up to a sign change is exact by solving \eqref{eq:eqminregulPRexact}. Furthermore, we provide an explicit expression of the recovery bounds for decomposable regularizers (that include the lasso, the group lasso, and their ordered weighted versions), for frame analysis-type regularizers and the total variation. To the best of our knowledge, these are the first results of this kind in the literature of phase retrieval. Our results encompass those of \cite{Voroninski14} as a special case. Some of our arguments in the proof can be seen as a generalization of those used compressed sensing (see \cite{chandrasekaran_convex_2012,vaiter2015low}) to the real phase retrieval problem.

Concerning stable recovery, we first consider a relaxed inequality constrained form \eqref{eq:eqminregulPRconst} which is known as the residual method or Mozorov formulation. We show that under the previous deterministic conditions, the set of solutions is non-empty. Moreover, the solutions are located in a ball of center $\avx$ up to a sign-change and radius equal to the signal-to-noise ratio. For standard Gaussian measurements and a large class of regularizers, we show with high probability that solving \eqref{eq:eqminregulPRconst} yields a solution that is near $\avx$ up to a sign change as soon as the number of measurements is large enough. 

We then turn to penalized problem \eqref{eq:eqminregulPRpen}. First, we show that under an appropriate geometric deterministic condition, the problem has a non-empty compact set of minimizers. Then, using $\Gamma-$convergence tools, when $\lambda\to0$ and $\epsi\to0$, we show that the set of minimizers converges to the set of true vectors up to a global sign change. Finally, we show that for enough small noise, the recovery error bound scales as $O(\normm{\epsi})$, a rate known in the inverse problem literature as linear convergence\footnote{The reason is that the bound is indeed linear in $\normm{\epsi}$.}. This holds under a non-degenerate source condition and a restricted injectivity condition tailored to the phase retrieval problem.

For standard Gaussian measurements, we provide sample complexity bounds for the latter two conditions to hold with high probability with several regularizers. This covers both the popular sparse retrieval case, but goes far beyond it by providing bounds that are new and unknown in the literature to the best of our knowledge. Our results include those of \cite{cai_optimal_2016,yang_misspecified_2019} as a special case. Some of our arguments and results generalized those in compressed sensing (see \cite{candes2011simple,vaiter2015low}) to the real phase retrieval problem.

\subsection{Outline of the paper}
The rest of this paper is organized as follows. Section~\ref{sec:prelim} gathers all the preliminaries required for our exposition. In Section~\ref{sec:unique}, we study exact recovery of the noiseless regularized phase retrieval problem. We then turn to stability to noise where we consider the the Mozorov (constrained) formulation in Section~\ref{sec:robustconstrained} and the Tikhonov (penalized) formulation in Section~\ref{sec:robustpenalized}. The proofs of the technical results are deferred to the appendix.

\section{Notations and preliminaries}\label{sec:prelim}
\paragraph{Vectors and matrices} 
We denote $\pscal{\cdot,\cdot}$ a scalar product on $\bbR^n$ and $\normm{\cdot}$ the corresponding norm. $\ball(x,r)$ is the corresponding ball of \tcb{radius $r$} centered at $x$ and $\mathbb{S}^{n-1}$ is the corresponding unit sphere. Moreover, $\norm{\cdot}{p}, p\in[1,\infty]$ stands for the $\ell_p$ norm. 

For $m \in \N^*$, we use the shorthand notation $\bbrac{m}=\{1,\ldots,m\}$.  The $i$-th entry of a vector $x$ is denoted $x[i]$. For any $y\in\bbR^m$ the operations $\abs{y}$ and $y^2$ should be understood componentwise. Given a matrix $M \in \bbR^{m \times n}$, $\transp{M}$ is its transpose. Let $\lambda_{\min}(M)$ and $\lambda_{\max}(M)$ be respectively the smallest and the largest eigenvalues of $M$. 

In the following, for a subspace $V \subset \bbR^n$, $\proj_V$ denotes the orthogonal projector on $V$, and
\begin{equation*}
x_V \eqdef \proj_V x \qandq M_V \eqdef M \proj_V.
\end{equation*}
For a matrix $M$ and index set $I$, $M^I$ (resp. $M_I$) denotes the sub-matrix whose rows (resp. columns)  are only those of $M$ indexed by $I$.

\paragraph{Sets}
For a finite subset $I$, $\abs{I}$ is its cardinality and $I^c$ its complement. For a non-empty set $\calS \in \bbR^n$, we denote $\cl{\calS}$ its closure,  $\clconv{\calS}$ the closure of its convex hull, and $\iota_\calS$ its indicator function, \ie $\iota_\calS(x)=0$ if $x \in \calS$ and $+\infty$ otherwise. For a non-empty convex set $\calS$, its \emph{affine hull} $\Aff(\calS)$ is the smallest affine manifold containing it. It is a translate of its \emph{parallel subspace} $\Lin(\calS)$, \ie $\Lin(\calS) = \Aff(\calS) - \calS= \bbR(\calS-\calS)$, for any $x \in \calS$. The \emph{relative interior} $\ri(\calS)$ of a convex set $\calS$ is the interior of $\calS$ for the topology relative to its affine full.\\

For any vector $x\in\bbR^n$, the distance to a non-empty set $\calS \subset \bbR^n$ is   
\begin{equation}
\dist(x,\calS) \eqdef \inf_{z \in \calS}\normm{x-z} .
\end{equation}

Throughout, we use the shorthand notation $\avX \eqdef \ens{\pm\avx}$ to denote the set of true vectors. Hence, for any vector $x\in\bbR^n$, the distance to the set of true vectors is  
\begin{equation}\label{eq:dist-true}
\dist(x,\avX)\eqdef \min\pa{\normm{x-\avx},\normm{x+\avx}} .
\end{equation}

\begin{remark}
Our limitation of the set of true solutions to $\{\pm\avx\}$ may appear restrictive since even for real vectors, the equivalence class is much larger than what we are allowing. However, this restriction will be justified later.
\end{remark}

\begin{definition}[Support function] The \emph{support function} of $\calS \subset \bbR^n$ is
\[
\sigma_{\calS}(z)=\sup_{x \in \calS} \pscal{z,x}.
\]
\end{definition}

\begin{definition}[Polar set]\label{defn:convex-polarset}
Let $\calS$ be a non-empty convex set. The set $\calS^\circ$ given by
\begin{equation*}
\calS^\circ = \enscond{v \in \bbR^n}{\pscal{v,x} \leq 1 \quad \forall x \in \calS}
\end{equation*}
is called the polar of $\calS$.
\end{definition}
The set $\calS^\circ$ is closed convex and contains the origin. When $\calS$ is also closed and contains the origin, then it coincides with its bipolar, \ie $\calS^{\circ\circ}=\calS$.

\begin{definition}[Gauge]\label{def:gauge-set} Let $\calS \subseteq \bbR^n$ be a non-empty closed convex set containing the origin. The gauge of $\calS$ is the function $\gauge_\calS$ defined on $\bbR^n$ by
\begin{equation*}
  \gauge_\calS(x) =
  \inf \enscond{\mu > 0}{x \in \mu \calS} .
\end{equation*}
As usual, $\gauge_\calS(x) = + \infty$ if the infimum is not achieved as a minimum.
\end{definition}
We have the following characterization of the support function in finite dimension. 
$\gauge_\calS$ is a non-negative, closed  and sublinear function. When $\calS$ is a closed convex set containing the origin, then 
\begin{equation*}
\gauge_\calS = \sigma_{\calS^\circ} \qandq \gauge_{\calS^\circ} = \sigma_\calS.
\end{equation*}
Let $\calS\subset \bbR^n$ a nonempty, closed bounded and convex subset. If $0 \in \ri(\calS)$, then $\gauge_\calS \in \Gamma_0(\R^n)$ is sublinear, non-negative and finite-valued, and $\sigma_{\calS}(x) = 0 \iff x \in \LinHull(\calS))^\bot$.

\begin{definition}[Asymptotic cone] 
Let $\calS$ be a non-empty closed convex set. The asymptotic cone, or recession cone is the closed convex cone defined by 
\[
\calS_{\infty}\eqdef\bigcap\limits_{t>0}\frac{\calS-x}{t},\quad x\in \calS. 
\]
\end{definition}
This definition  does not depend on the choice of $x\in\calS$. 
The importance of the asymptotic cone becomes obvious through the following fundamental fact that will then play a crucial role to study well-posedness of our minimization problems; see \cite[Proposition~2.1.2]{AT03}.

\begin{fact}
$\calS$ is compact if and only if $\calS_{\infty}=\Ba{0}$.
\end{fact} 

\paragraph{Functions}
A function  $f: \bbR^n \to \barR = \bbR \cup \{\pinf\}$ is closed (or lower semicontinuous (lsc)) if so is its epigraph. The effective domain of $f$ is $\dom(f) = \enscond{x\in\bbR^n}{f(x) < +\infty}$ and $f$ is proper if $\dom(f) \neq \emptyset$ as is the case when it is finite-valued. We denote by $\Gamma_0(\bbR^n)$ the class of proper lsc convex function on $\bbR^n$.  

A function $f$ is coercive if $\lim_{\normm{x} \to +\infty}f(x)=+\infty$. $f$ is said sublinear if it is convex and positively homogeneous. The Legendre-Fenchel conjugate of $f$ is $f^*(z)=\sup_{x \in \bbR^n} \pscal{z,x} - f(x)$.
Let the kernel of a function be defined as $\ker(f)\eqdef\Ba{z\in\bbR^n:f(z)=0}$.
Let us denote by $\calS_{\lev f}(\avx)\eqdef \Ba{z\in\bbR^n:f(z)\leq f(\avx)}$ the sublevel set of $f$ at $\avx$.

The subdifferential $\partial f(x)$ of a convex function $f: \bbR^n \to \barR$ at $x$ is the set
\[
\partial f(x) = \enscond{v \in \bbR^n}{f(z) \geq f(x) + \pscal{v,z-x}, \quad \forall z \in \dom(f)} .
\]
An element of $\partial f(x)$ is called a subgradient. If the convex function $f$ is differentiable at $x$, then its only subgradient is its gradient, \ie $\partial f(x) = \ens{\nabla f(x)}$.

The Bregman divergence associated to a convex function $f$ at $x$ with respect to $v \in \partial f(x) \neq \emptyset$ is
\[
\breg{f}{v}{z}{x} = f(z) - f(x) - \pscal{v,z-x} .
\]
The Bregman divergence is in general non-symmetric. It is also non-negative by convexity. When $f$ is differentiable at $x$, we simply write $\breg{f}{}{z}{x}$ (which is, in this case, also known as the Taylor distance).

For a proper closed function $f: \bbR^n \to \barR$, $f_\infty: \bbR^n \to \barR$ is the asymptotic function or recession function \cite{AT03} associated with $f$. It is defined by
\begin{equation}\label{eq:AF}
f_\infty(z) \eqdef \liminf_{z^\prime \to z, t \to +\infty}\frac{f(t z^\prime)}{t} .
\end{equation}
It is well-known that $f_\infty$ is lsc and positively homogeneous and that its epigraph is the asymptotic cone of the epigraph of $f$. This function plays an important role in the existence of solutions to minimization problems. Besides for any closed convex set $\calS$, one has 
$
\pa{\iota_{\calS}}_{\infty}=\iota_{\calS_{\infty}}. 
$

\paragraph{Operator norm} Let $g_1$ and $g_2$ be two finite-valued gauges  defined on two vector spaces $V_1, V_2$, and $A:V_1\to V_2$ be a linear map. The  \textit{operator bound} $\anormOP{A}{g_1}{g_2}$ of $A$ between $g_1$ and $g_2$ is given by
\[
\anormOP{A}{g_1}{g_2}=\sup\limits_{g_1(x)\leq1} g_2(Ax).
\]
Let us note that $\anormOP{A}{g_1}{g_2}<\infty$ if and only if $A\ker(g_1)\subset \ker(g_2).$ Moreover a sufficient condition for $\anormOP{A}{g_1}{g_2}<\infty$  is that $g_1$ is coercive. As a convention, $\anormOP{A}{g_1}{\norm{\cdot}{p}}$ is denoted  as $\anormOP{A}{g_1}{p}$. A direct consequence of this definition is the fact that, for every $x\in V_1$, 
\[
g_2(Ax)\leq \anormOP{A}{g_1}{g_2} g_1(x).
\]

\paragraph{Preliminaries from probability theory}

Many of the following notations for probabilistic concepts are adopted directly from \cite{vershynin_introduction_2011,tropp_user-friendly_2012}.  We denote by $\prspace$ a \emph{probability space}. 

\begin{definition}\label{def:net}
	Let $\calS$ be an arbitrary bounded subset of $\bbR^n$. The covering number of $\calS$ in the
	norm $\normm{\cdot}$ at resolution $\delta > 0$ is the smallest number, $N(\calS,\delta)$, such that $\calS$ can be covered with balls $\ball(x_i,\delta)$, $x_i \in \calS$, $i\in \bbrac{N(\calS,\delta)}$, \ie,
	\[
	\calS \subseteq \bigcup_{i \in \bbrac{N(\calS,\delta)}} \ball(x_i,\delta)
	\]
	The finite set of points $\calS_\delta \eqdef \enscond{x_i}{i \in \bbrac{N(\calS,\delta)}}$ is called a (proper) $\delta$-covering of $\calS$. The packing number $P(\calS,\delta)$ is the maximal integer such that there are points $x_i \in \calS$, $i \in \bbrac{P(\calS,\delta)}$, such that $\normm{x_i-x_j} > \delta$ for all $i,j \in \bbrac{P(\calS,\delta)}$, $i \neq j$. The set of such points is called a $\delta$-net of $\calS$.
\end{definition}
Packing and covering numbers are closely related as one always has
\begin{equation}\label{eq:packcovnumb}
P(\calS,2\delta) \leq N(\calS,\delta) \le P(\calS,\delta) .
\end{equation}

\begin{definition}\label{def:gwidth}
	The Gaussian width of a subset $\calS \subset \bbR^n$ is defined as
	\[
	w(\calS) \eqdef \esp{\sigma_{\calS}(Z)}, \qwhereq Z \sim \calN(0,\Id_n) .
	\]
\end{definition}
The Gaussian width is a summary geometric quantity that, informally speaking, measures the size of the bulk of a set in $\bbR^n$. This concept plays a central role in high-dimensional probability and its applications. It has appeared in the literature in different contexts \cite{gordon1988}. In particular, it has been used to establish sample complexity bounds to ensure exact recovery (noiseless case) and mean-square estimation stability (noisy case) for low-complexity penalized estimators from linear Gaussian measurements; see e.g.~\cite{chandrasekaran_convex_2012,amelunxen_living_2014,oymak_squared-error_2013,vaiterimaiai13}. The Gaussian width has deep connections to convex geometry and it enjoys many useful properties. It is well-known that it is positively homogeneous, monotonic w.r.t inclusion, and invariant under orthogonal transformations. Moreover, one has 
\[w(\calS)=w(\cl{\calS})=w(\conv{\calS})=w(\clconv{\calS}).\]
This comes from the properties of the support function. A lower bound for the Gaussian width of a bounded set can be obtained via Sudakov's minoration.
\begin{proposition}\label{prop:Sudakov}
	Let $\calS$ be a bounded set. Then for any $\delta > 0$ small enough, we have
	\[
	w(\calS) \geq \frac{\delta}{2} \sqrt{\log\Ppa{N(\calS,\delta)}}.
	\]
\end{proposition}

\begin{proof}
	Let $\calS_\delta$ be an $\delta$-net of $\calS$. Thus, since $\calS_{\delta} \subset \calS$, we have
	\[
	w(\calS) \geq w(\calS_{\delta}) = \esp{\max_{x_i \in \calS_\delta} \pscal{Z,x_i}} .
	\]
	Since $\transp{[\pscal{Z,x_i}: i \in \bbrac{P(\calS,\delta)}]}$ is a zero-mean Gaussian vector, we can invoke the lower bound in \cite[Theorem~13.4]{BoucheronLugosi} to get
	\begin{align*}
	w(\calS) 
	&\geq \esp{\max_{x_i \in \calS_\delta} \pscal{Z,x_i}} \\
	&\geq \frac{1}{2} \min_{i \neq j, x_i,x_j \in \calS_\delta} \sqrt{\esp{|\pscal{x_i-x_j,Z}|^2}\log\Ppa{P(\calS,\delta)}} \\
	&= \frac{1}{2} \min_{i \neq j, x_i,x_j \in \calS_\delta} \normm{x_i-x_j} \sqrt{\log\Ppa{P(\calS,\delta)}} \\
	&\geq \frac{\delta}{2} \sqrt{\log\Ppa{N(\calS,\delta)}} ,
	\end{align*}
	where we used and \eqref{eq:packcovnumb} and the definition of the packing number in the last inequality. 
\end{proof}

%
%
%
%

The following proposition gives the concentration of measure in the Gauss space. A comprehensive account can be found in \cite{Ledoux2005}.
\begin{proposition}\label{pro:gauss-space}Let $f$ be a real-valued $K-$Lipschitz continuous on $\bbR^n$. Let $Z$ be the standard normal random vector in $\bbR^n$. Then  for every $t\geq0$ one has
	\[
	\prob\Ppa{f(Z)-\esp{f(Z)}\geq t}\leq \exp(-t^2/2K^2).
	\]
\end{proposition}


\section{Noiseless Recovery}\label{sec:unique}
We here study well-posedness (existence and uniqueness of minimizers) of \eqref{eq:eqminregulPRexact}, which in turn will allow us to state when exact recovery is possible. In this section, we use the shorthand notation 
\[
\calS_{\avy,0} \eqdef \Argmin_{A^{-1}\Ppa{\overbar{\calF}}}(R) .
\]

\subsection{Existence and compactness}\label{sec:uniquenoiseless}

The following result provides sufficient conditions under which problem \eqref{eq:eqminregulPRexact} is well-posed. It does not need convexity of $R$.
\begin{proposition}\label{prop:existencenoiseless}
Let $R: \bbR^n \to \barR$ be a proper and lsc function. Assume that:
\begin{enumerate}[label=(\roman*)] 
\item $A(\dom(R)) \cap \overbar{\calF} \neq \emptyset$. \label{H:domRA}
\item $R$ is non-negative\footnote{In fact, we need $R$ to be only bounded from below, and there is no loss of generality by taking the lower bound as $0$ by a trivial translation argument.}. \label{H:Rbnd}
\item $\Ker(R_\infty) \cap \Ker(A) =\{0\}$. \label{H:RinfA}
\end{enumerate}
Then $\calS_{\avy,0}$ is a non-empty compact set. 
\end{proposition}
\begin{remark}\label{rmk:existence}{\ }
\begin{itemize}
    \item A typical case where all above assumptions are in force is when $R$ is coercive, has full domain and is bounded from below.
    \item  This result is general and goes beyond the phase retrieval problem, indeed this result can be applied  for instance for general non-linear inverse problems with a suitable definition of $\overbar{\calF}$.
\end{itemize}    
\end{remark}

\begin{proof}
The range of $R_\infty$ is on $\bbR_+$ since $R$ verifies \ref{H:Rbnd}. Define $G = R + \iota_{\overbar{\calF}} \circ A$. In view of the domain qualification assumption \ref{H:domRA}, we get by \cite[Proposition~2.6.1 and  Proposition 2.6.3]{AT03} that
\[
G_\infty(z) \geq R_\infty(z) + \iota_{\overbar{\calF}_\infty}(Az) .
\]
Since $\overbar{\calF}$ is bounded, we get that $\overbar{\calF}_\infty = \ens{0}$. Moreover, the range of $R_{\infty}$ is on $\bbR_+$ since $R$ is bounded from below. Thus
\[
G_\infty(z) > 0 \qforallq z \notin \Ker(R_\infty) \cap \Ker(A) .
\]
It then follows from \cite[Corollary~3.1.2]{AT03} that \ref{H:RinfA} entails the claim. 
\end{proof}

\subsection{Deterministic recovery condition}\label{sec:recoverynoiselessdet}

\begin{definition}[Descent cone] The descent cone of $R$ at $\avx$ is the conical hull of the sublevel set of $R$ at $\avx$, \ie 
\begin{equation}\label{eq:descone}
\calD_R(\avx) \eqdef \bigcup_{t > 0}\enscond{z}{R(\avx+t z) \leq R(\avx)} .
\end{equation}
\end{definition}
The tangent cone of the sublevel set of $R$ at $\avx$, denoted $\calT_R(\avx)\eqdef\clcone\pa{\calS_{\lev R}(\avx)-\avx}$, is the closure of $\calD_R(\avx)$. The normal cone of the sublevel set of $R$ at $\avx$ is 
\[
\calN_R(\avx)\eqdef\Ba{s:\pscal{s,z-\avx}\leq0,z\in\calS_{\lev R}(\avx)},
\]
and we have $\calN_R(\avx) = \calT_R(\avx)^\circ$, where we recall that $^\circ$ stand for polarity (see Definition~\ref{defn:convex-polarset}).

\begin{theorem}\label{thm:exactrecoverydet}
Suppose that $\calS_{\avy,0} \neq \emptyset$, and that:
\begin{enumerate}[label=\rm({\bf{H.\arabic*}})] 
\item $R \in \Gamma_0(\bbR^n)$ and is even symmetric. \label{H:Reven}
\item \label{H:injdescent}
$\forall I \subset \bbrac{m}, |I| \geq m/2$ 
\[
\Ker(A^I) \cap \calD_R(\avx) = \ens{0} .
\]
\end{enumerate}
Then the recovery of $\avx$ (up to a global sign) is exact by solving \eqref{eq:eqminregulPRexact}, \ie   
\[
\calS_{\avy,0} = \avX .
\]
\end{theorem}

\begin{remark}{\ }
\begin{itemize}
    \item Assumption \ref{H:Reven} is quite general and encompasses any convex symmetric gauge. This includes the $\ell_1$, $\ell_{1,2}$ and $\ell_{\infty}-$norm regularizers as well as their analysis-type counterparts.
    \item Of course, assumption \ref{H:injdescent} is vacuous if $\calD_R(\avx)$ is empty, which is the case if the set of minimizers is empty. The assumptions of Proposition~\ref{prop:existencenoiseless} ensure that this cannot be the case.
\end{itemize}
\end{remark}

\begin{proof}
The proof is a generalization of that \cite[Theorem~2.2]{Voroninski14} beyond the $\ell_1$-norm, and exploits the structure of the constraint set $\overbar{\calF}$. Let $b \eqdef A\avx$, and for any sign vector $\eps \in \ens{1,-1}^m$, set $b_\eps \eqdef \transp{[\eps[r]b[r]: \; r \in \bbrac{m}]}$. Consider the minimization problem
\[
\inf_{x \in \bbR^n} R(x) \qsubjq Ax = b_\eps ,
\]
and denote $x_\eps$ any minimizer, if it exists. If $x_\eps$ does not exist, there is nothing to say. We claim that if $x_\eps$ exists, then under our assumptions, for any sign vector $\eps$,
\[
R(\avx) \leq R(x_\eps)
\]
with equality iff $x_\eps = \pm \avx$.

Observe that $x_\eps \in A^{-1}(\overbar{\calF})$. Thus $\pscal{a_r,x_\eps} = \pm b[r]$ for all $r \in \bbrac{m}$. Let
\[
I = \enscond{r \in \bbrac{m}}{\pscal{a_r,x_\eps} = b[r]} .
\]
Thus either $|I| \geq m/2$ or $|I^c| \geq m/2$. Assume the first case holds. This implies that $A^I x_\eps = A^{I} \avx$. From \cite[Proposition~2.1]{chandrasekaran_convex_2012}, it follows using \ref{H:injdescent} and convexity of $R$ that
\[
\Argmin_{x \in \bbR^n} \ens{R(x) \tsubjt A^I x = A^{I} \avx} = \ens{\avx} ,
\]
and thus, since $x_\eps$ is a feasible point,
\[
R(\avx) \leq R(x_\eps) ,
\]
with equality holding if and only if $x_\eps = \avx$. For the case where $|I^c| \geq m/2$, we have $-A^{I^c} x_\eps = A^{I^c} \avx$. Arguing similarly as before using also that $R$ is even, we get
\[
\Argmin_{x \in \bbR^n} \ens{R(x) \tsubjt -A^{I^c} x = A^{I^c} \avx} = 
-\Argmin_{x \in \bbR^n} \ens{R(x) \tsubjt A^{I^c} x = A^{I^c} \avx} = \ens{-\avx} ,
\]
Thus, in this case
\[
R(\avx) \leq R(x_\eps) ,
\]
with equality holding if and only if $x_\eps = -\avx$. Since this holds for any $\eps \in \ens{1,-1}^m$ and any minimizer of \eqref{eq:eqminregulPRexact} is of the form $x_\eps$ (when the latter exists), we conclude.
\end{proof}

\subsection{Recovery from Gaussian measurements}\label{sec:recoverynoiselessgaussian}
The goal now is to give sample complexity bounds for the claims of Theorem~\ref{thm:exactrecoverydet} to hold true when $A$ is a Gaussian measurement map, \ie the entries of $A$ are \iid $\calN(0,1/m)$. We start with the following preparatory lemma.

\begin{lemma}\label{lem:sigAI}
Let $A: \bbR^n \to \bbR^m$ be a Gaussian map with \iid $\calN(0,1/m)$ entries. Let $\delta \in ]0,1[$ and $\nu = \frac{1}{18}\sqrt{\frac{\pi}{2}}$. Suppose that $x \in \bbR^n$ is a fixed vector. Then 
\[
\min_{I \subset \bbrac{m}, |I| \geq m/2} \normm{A^{I} x} \geq \nu/2\normm{x}
\]
with probability at least $1-2e^{-\frac{\nu^2m}{8}}$, and 
\[
\max_{I \subset \bbrac{m}, |I| \geq m/2} \normm{A^{I} x} \leq (1+\delta)\normm{x}
\]
with probability at least $1-e^{-\frac{\delta^2m}{2}}$.
\end{lemma}

\begin{proof}
The first claim follows from \cite[Lemma~4.4]{Voroninski14}. The second one follows from the fact that
\[
\normm{A^{I} x} \leq \normm{A x} \qforallq I \subset \bbrac{m} ,  
\]
and then use Proposition~\ref{pro:gauss-space} since $A \mapsto \normm{A x}$ is $\normm{x}$-Lipschitz continuous and $\esp{\normm{Ax}} \leq \normm{x}/\sqrt{m}$.
\end{proof}

\begin{theorem}\label{thm:exactrecoverygaussian}
Suppose that \ref{H:Reven} holds. Let $\nu$ be as defined in Lemma~\ref{lem:sigAI}. Let $A: \bbR^n \to \bbR^m$ be a Gaussian map with \iid $\calN(0,1/m)$ entries such that
\[
m \geq \frac{8(1+t)}{\nu^2} \log\Ppa{N\Ppa{\calD_R(\avx) \cap \bbS^{n-1},\eps}} ,
\]
for some $\eps \in ]0,\nu/(2+\nu)[$ and $t > 0$. Then with probability at least $1-3e^{-\frac{t\nu^2m}{8}}$, the recovery of $\avx$ (up to a global sign) is exact by solving \eqref{eq:eqminregulPRexact}.
\end{theorem}

\begin{proof}
The proof relies on combining Theorem~\ref{thm:exactrecoverydet} and Lemma~\ref{lem:sigAI} together with a covering argument. Throughout the proof, denote $\Omega = \calD_R(\avx) \cap \bbS^{n-1}$. In view of Theorem~\ref{thm:exactrecoverydet}, we need to prove that there exists $c \in ]0,1[$ such that
\[
\min_{I \subset \bbrac{m}, |I| \geq m/2} \normm{A^{I} z} \geq c
\]
for all $z \in \Omega$. Let $\Omega_\eps = \enscond{W_i}{i \in \bbrac{N\Ppa{\Omega,\eps}}}$ be an $\eps$-covering of $\Omega$. For a fixed $W_i \in \Omega_\eps$, Lemma~\ref{lem:sigAI} tells us that
\[
\normm{A^{I} W_i} \geq \nu/2
\]
with probability at least $1-2e^{-\frac{\nu^2m}{8}}$. Now, for an arbitrary but fixed $z \in \Omega$, there exists $W_j \in \Omega_\eps$ such that $\normm{z-W_j} \leq \eps$. Thus
\[
\min_{I \subset \bbrac{m}, |I| \geq m/2} \normm{A^{I} z} \geq \min_{I \subset \bbrac{m}, |I| \geq m/2} \normm{A^{I} W_j} - \max_{I \subset \bbrac{m}, |I| \geq m/2} \normm{A^{I} (z-W_j)} \geq \frac{\nu}{2} - \Ppa{1+\frac{\nu}{2}}\eps
\]
with probability at least $1-3e^{-\frac{\nu^2m}{8}}$, where we took  $\delta=\nu/2$ in Lemma~\ref{lem:sigAI} for the second inequality. Taking $\eps$ small as devised, we deduce that
\begin{equation}\label{eq:sminconcent}
\min_{I \subset \bbrac{m}, |I| \geq m/2} \normm{A^{I} z} \geq \frac{\nu}{2} - \Ppa{1+\frac{\nu}{2}}\eps \in ]0,\nu/2[
\end{equation}
holds for all $z \in \Omega$ with probability at least $1-3e^{\log\Ppa{N\Ppa{\Omega,\eps}}-\frac{\nu^2m}{8}}$.
The bound on the number of measurements then leads to the claim.
\end{proof}

Estimating covering numbers is difficult for general convex cones. On the other hand, the authors in \cite{chandrasekaran_convex_2012,amelunxen_living_2014,tropp2015convex} developed a general recipe for estimating Gaussian widths of the descent cone (restricted to the unit sphere). This is what we will do in Section~\ref{sec:recoverynoiselessgaussiandecomp}. But before, we need a sample complexity bound in terms of the Gaussian width. This is the motivation behind the following corollary.

\begin{corollary}\label{cor:exactrecoverygwidth}
Suppose that \ref{H:Reven} holds. Let $\nu$ be as defined in Lemma~\ref{lem:sigAI}. Let $A: \bbR^n \to \bbR^m$ be a Gaussian map with \iid $\calN(0,1/m)$ entries such that
\[
m \geq \frac{64(1+t)(\nu+2)^2}{\nu^4} w\Ppa{\calD_R(\avx) \cap \bbS^{n-1}}^2
\]
for $t > 0$. Then with probability at least $1-3e^{-\frac{t\nu^2m}{8}}$, the recovery of $\avx$ (up to a global sign) is exact by solving \eqref{eq:eqminregulPRexact}.
\end{corollary}

\begin{proof}
Use the lower bound of Proposition~\ref{prop:Sudakov} and choose $\eps=\frac{\nu}{\sqrt{2}(2+\nu)}$ in  Theorem~\ref{thm:exactrecoverygaussian}.
\end{proof}
\begin{remark}{\ }
\begin{enumerate}[label=(\roman*)]
    \item Clearly, this result shows that the sample complexity bound for exact phase recovery by solving \eqref{eq:eqminregulPRexact} is nearly (up to constants) the same as for exact recovery from linear Gaussian measurements (\ie compressed sensing) \cite{chandrasekaran_convex_2012,amelunxen_living_2014}. However, one has to keep in mind that \eqref{eq:eqminregulPRexact} contains a non-convex constraint, which is not algorithmically tractable, and the recovery results we have are not for an algorithmic scheme.
    \item Unlike the linear case, our results for quadratic measurements above cannot be extended to sub-Gaussian random sensing vectors. The most technical reason is that the lower bound in Lemma~\ref{lem:sigAI}, which uses \cite[Lemma~4.4]{Voroninski14}, heavily relies on the Gaussian structure of the sensing vectors. As a simple counter-example (see also \cite[Remark~2.3]{Voroninski14}), consider the recovery of any vector of the standard basis from Rademacher measurements, \ie the entries of $A$ are \iid taking values $\pm 1$ with probability $1/2$. Then all such vectors have the same image under $|A \cdot|$. Thus regardless of the number of measurements, even such simple one-sparse vectors cannot be reconstructed from Rademacher sensing vectors.   
\end{enumerate}  
\end{remark}



\subsection{Recovery bounds for decomposable regularizers}\label{sec:recoverynoiselessgaussiandecomp}
We start by defining some essential geometrical objects associated to the non-smoothness structure of the regularizer $R$ at a given vector $x$, as introduced in \cite{vaiterimaiai13}. These objects capture the model structure undelying $x$. 

\begin{definition}[Model Subspace]\label{defn:linmod}
  Let $x \in \bbR^n$. We define
  \begin{equation*}
    \e{x} \eqdef \proj_{\Aff(\partial R(x))}(0).
  \end{equation*}
  We also define 
  \begin{equation*}
    \S_{x} \eqdef \Lin (\partial R(x)) \qandq \T_{x} \eqdef \S_{x}^\bot .
  \end{equation*}
  $\T_{x}$ is coined the \emph{model subspace} of $x$ associated to $R$.
\end{definition}
It can be shown, see \cite[Proposition~5]{vaiterimaiai13}, that $x \in \T_{x}$, hence the name model subspace. When $R$ is differentiable at $x$, we have $\e{x}=\nabla R(x)$ and $\T_{x}=\bbR^n$. When $R$ is the $\ell_1$-norm (Lasso regularizer), the vector $\e{x}$ is nothing but the sign of $x$. Thus, $\e{x}$ can be viewed as a generalization of the sign vector. Observe also that $\e{x}=\proj_{\T_{x}}(\partial R(x))$, and thus $\e{x} \in \T_{x} \cap \Aff(\partial R(x))$. However, in general, $\e{x} \not\in \partial R(x)$.

\medskip

In this subsection, we will assume that $R$ is a strong gauge in the sense of \cite[Definition~6]{vaiterimaiai13}. 
\begin{definition}[Strong Gauge]\label{def:Rstrong} 
$R$ is a strong gauge if $R=\gamma_{\calC}$
where $\calC$ is a non-empty convex compact set containing the origin as an interior point, and $\e{x} \in \ri(\partial R(x))$. 
\end{definition}
Strong gauges have a nice decomposable description of $\partial R(x)$ in terms of $\e{x}$, $\T_{x}$, $\S_{x}$ and $\sigma_{\calC}$. More precisely, piecing together \cite[Theorem~1, Proposition~4 and Proposition~5(iii)]{vaiterimaiai13}, we have
\begin{equation}\label{eq:decompstrong}
\partial R(x) = \Aff(\partial R(x)) \cap \calC^\circ =
\enscond{v \in \bbR^n}
{
  v_{\T_{x}} = \e{x}
  \qandq
  \sigma_{\calC}(v_{\S_{x}}) \leq 1
}.
\end{equation}
The Lasso, group Lasso, and nuclear norms are typical popular examples of (symmetric) strong gauges. Let us observe that strong symmetric gauges not only conform to \ref{H:Reven} but also meet the requirements of Proposition~\ref{prop:existencenoiseless}.

\medskip

The following result provides a useful upper-bound on the Gaussian width of the descent cone of a strong gauge in terms of $\e{x}$, $\T_{x}$, $\S_{x}$ and $\sigma_{\calC}$.
\begin{lemma}\label{lem:gwidthbndgauge}
If $R$ is a strong gauge of $\calC$, then for any $x \in \bbR^n \setminus \ens{0}$
\begin{equation}\label{eq:gwidthbndgauge}
w\Ppa{\calD_R(x) \cap \bbS^{n-1}}^2 \leq \esp{\sigma_{\calC}(Z_{\S_x})^2}\normm{\e{x}}^2 + \dim\Ppa{\T_{x}} , \qquad Z \sim \calN(0,\Id_n) . 
\end{equation}
\end{lemma}

Clearly this upper-bound scales linearly in the intrinsic dimension of $x$ via the dimension of the model space $T_x$. Thus, one expects Corollary~\ref{cor:exactrecoverygwidth} to provide us with a complexity bound that scales linearly in $\dim(T_{\avx})$ to recover $\pm\avx$ by solving \eqref{eq:eqminregulPRexact}. This is what we will show shortly for a few popular regularizers.

\begin{proof}
From \cite[Proposition~3.6]{chandrasekaran_convex_2012} and Jensen's inequality, we have
\[
w\Ppa{\calD_R(x) \cap \bbS^{n-1}}^2 \leq \esp{\dist\Ppa{Z,\calD_R(x)^\circ}^2} = \esp{\dist\Ppa{Z,\calN_R(x)}^2} .
\]
$R$ being a strong gauge implies that $R$ is convex and has full domain, and thus $\partial R$ is non-empty convex- and compact-valued at any $x \in \bbR^n$. Moreover, $\Argmin(R) = \ens{0}$. It then follows from \cite[Theorem~23.7]{rockafellar_convex_1970} that for any $x \neq 0$
\[
\calN_R(x)= \bigcup_{t \geq 0}t\partial R(x) ,
\]
where $t\partial R(x)$ is the dilation of the subdifferential through the scaling factor $t$.
In turn, we get
\[
w\Ppa{\calD_R(x) \cap \bbS^{n-1}}^2 \leq \esp{\dist\Ppa{Z,\cup_{t \geq 0}t\partial R(x)}^2} \leq \inf_{t \geq 0} \esp{\dist\Ppa{Z,t\partial R(x)}^2} \leq \esp{\dist\Ppa{Z,\tilde{t}\partial R(x)}^2}
\]
for any $\tilde{t} \geq 0$. Observe that in view of definition~\eqref{eq:decompstrong}, we have
\begin{equation}\label{eq:t-dilation}
    t\partial R(x) = 
\enscond{v \in \bbR^n}
{
  v_{\T_{x}} = t\e{x}
  \qandq
  \sigma_{\calC}(v_{\S_{x}}) \leq t
}.
\end{equation}
We will now device an appropriate choice of $\tilde{t}$ and of a subgradient in $\partial R(x)$\footnote{This generalizes the reasoning of \cite{rao2012signal} beyond group sparsity.}. Let $v$ be a (random) vector such that $v_{\S_x} = Z_{\S_x}$ and $v_{\T_x} = \sigma_{\calC}(Z_{\S_x})\e{x}$. Obviously, $v \in \sigma_{\calC}(Z_{\S_x})\partial R(x)$ by \eqref{eq:t-dilation}. Thus
\begin{align*}
w\Ppa{\calD_R(x) \cap \bbS^{n-1}}^2
&\leq \esp{\normm{Z-v}^2} \\
&= \esp{\normm{(Z_{\T_x}-v_{\T_x}) + (Z_{\S_x}-v_{\S_x})}^2} \\
&= \esp{\normm{Z_{\T_x}-\sigma_{\calC}(Z_{\S_x})\e{x}}^2} \\
&= \esp{\sigma_{\calC}(Z_{\S_x})^2}\normm{\e{x}}^2 + \esp{\normm{Z_{\T_x}}^2} \\
&= \esp{\sigma_{\calC}(Z_{\S_x})^2}\normm{\e{x}}^2 + \dim(\T_{x}) ,
\end{align*}
where we used orthogonality of $\T_{x}$ and $\S_{x}$ in the first equality and  $\esp{\normm{Z_{\T_x}}^2} = \tr(\proj_{\T_x}) = \dim(\T_{x})$ in the last equality. In the third equality, we used again orthogonality of $\T_{x}$ and $\S_{x}$ which entails that $Z_{\T_x}$ and $Z_{\S_x}$ are independent as $Z$ is zero-mean Gaussian.
\end{proof}

\paragraph{Lasso ($\ell_1$-norm)}
It is widely known that the $\ell_1$-norm promotes sparsity, see~\cite{BuhlmannVandeGeerBook11} for a comprehensive treatment. In our case, this corresponds to choosing
\begin{equation}
\label{eq:Rlasso}
R(x) = \norm{x}{1} = \sum_{i=1}^n \abs{x[i]} .
\end{equation}
This regularizer is also referred to as $\ell_1$-synthesis in the signal processing community.

\medskip

We denote $(a_i)_{1 \leq i \leq n}$ the standard basis of $\bbR^n$ and $\supp(x) \eqdef \enscond{i \in \bbrac{n}}{x[i] \neq 0}$. Then, we have (see \cite{vaiterimaiai13})
\begin{equation}\label{eq:exlasso}
\T_x = \Span \ens{(a_i)_{i \in \supp(x)}}, \quad 
\e{x}[i] = 
\begin{cases} 
\sign(x[i]) 	& \text{if~} i \in \supp(x) \\
0				& \text{otherwise}
\end{cases},
\qandq \sigma_{\calC}=\norm{\cdot}{\infty} .
\end{equation}
Thus if $\avx$ is $s$-sparse, \ie $\abs{\supp(\avx)}=s$, then $\dim(\T_{\avx}) = s$ and $\normm{\e{\avx}}^2 = s$. Moreover
\[
\esp{\sigma_{\calC}(Z_{\S_x})^2} = \esp{\max_{i \in \supp(\avx)^c} |Z[i]|^2} ,
\]
which is the expectation of the maximum of $(n-s)$ $\chi^2$-random variables with $1$ degree of freedom. We then have, using \cite[Example~2.7]{BoucheronLugosi} (see also \cite[Lemma~3.2]{rao2012signal}), that
\[
\esp{\max_{i \in \supp(\avx)^c} |Z[i]|^2} \leq \Ppa{\sqrt{2\log(n-s)} + 1}^2 .
\]
Plugging this into Lemma~\ref{lem:gwidthbndgauge} and using Corollary~\ref{cor:exactrecoverygwidth}, we obtain the following result.
\begin{proposition}\label{prop:exactrecoverygl1}
Let $\avx$ be an $s$-sparse vector. Let $A: \bbR^n \to \bbR^m$ be a Gaussian map with \iid $\calN(0,1/m)$ entries such that 
\[
m \geq \frac{64(1+t)(\nu+2)^2}{\nu^4} s\Ppa{\Ppa{\sqrt{2\log(n-s)} + 1}^2+1} 
\]
for some $t > 0$. Then with probability at least $1-3e^{-\frac{t\nu^2m}{8}}$, the recovery of $\avx$ (up to a global sign) is exact by solving \eqref{eq:eqminregulPRexact} with $R=\norm{\cdot}{1}$.
\end{proposition}
\begin{remark}
Clearly, $m \gtrsim s\log(n-s) + s$ measurements are sufficient for the exact recovery of an $s$-sparse vector from $m$ phaseless (actually signless) measurements of a real Gaussian map $A$. This can be improved to $m \gtrsim s\log(n/s) + s$ by exploiting the particular form of the normal cone of the $\ell_1$ norm, see \cite[Proposition~3.10]{chandrasekaran_convex_2012}. This leads to a measurement bound similar to the one in \cite{Voroninski14}. Note however that their recovery guarantee is RIP-based, and thus is uniform over all $s$-sparse vectors while our recovery analysis is non-uniform.
\end{remark}

\paragraph{Group Lasso ($\ell_1-\ell_2$ norm)}\label{par:l12norm}
The $\ell_1-\ell_2$ norm (a.k.a. group Lasso) is widely advocated to promote group/block sparsity, \ie it drives all the coefficients in one group to zero together hence leading to group selection, see~\cite{BuhlmannVandeGeerBook11}. The group Lasso regularization with $L$ groups reads
\begin{equation}
\label{eq:Rglasso}
R(x) = \norm{x}{1,2} \eqdef \sum_{i=1}^L \norm{x[b_i]}{2} .
\end{equation}
where $\bigcup\limits_{i=1}^L b_i = \bbrac{n}$, $b_i, b_j \subset \bbrac{n}$, and  $b_i \cap b_j = \emptyset$ whenever $i \neq j$. Define the group support as $\bsupp(x) \eqdef \enscond{i \in \bbrac{L}}{x[b_i] \neq 0}$.
Thus, one has
\begin{equation}\label{eq:exglasso}
\T_x = \Span \ens{(a_j)_{\enscond{j}{\exists i \in \bsupp(x), j \in b_i}}}, 
\quad \e{x}[b_i] = 
\begin{cases} 
\tfrac{x[b_i]}{\norm{x[b_i]}{2}} & \text{if~} i \in \bsupp(x) \\
0 				& \text{otherwise}
\end{cases}, 
\end{equation}
and 
\begin{equation}\label{eq:exglasso-g}
\sigma_{\calC}(v) = \max_{i \in \bsupp(\avx)^c} \norm{v[b_i]}{2} .
\end{equation}
Thus if $\avx$ is $s$-block sparse, \ie $\abs{\bsupp(\avx)}=s$, and the groups have equal size $B$, we have $\dim(\T_{\avx}) = sB$ and $\normm{\e{\avx}}^2 = s$. Moreover, \cite[Example~2.7]{BoucheronLugosi} yields
\[
\esp{\max_{i \in \bsupp(\avx)^c} \normm{Z[b_i]}^2} \leq \Ppa{\sqrt{2\log(L-s)} + \sqrt{B}}^2 .
\]
Hence, we get the following result for the group Lasso. 
\begin{proposition}\label{prop:exactrecoverygl1l2}
Let $\avx$ be an $s$-block sparse vector. Let $A: \bbR^n \to \bbR^m$ be a Gaussian map with \iid $\calN(0,1/m)$ entries such that 
\[
m \geq \frac{64(1+t)(\nu+2)^2}{\nu^4} s\Ppa{\Ppa{\sqrt{2\log(L-s)} + \sqrt{B}}^2+B} 
\]
for some $t > 0$. Then with probability at least $1-3e^{-\frac{t\nu^2m}{8}}$, the recovery of $\avx$ (up to a global sign) is exact by solving \eqref{eq:eqminregulPRexact} with $R=\norm{\cdot}{1,2}$.
\end{proposition}
\begin{remark} 
Our complexity bound for the group Lasso is of order $m\gtrsim s\Ppa{2\log(L-s)+ B}$. This bound is, up to a multiplicative constant, similar to the (linear) compressed sensing case with Gaussian sensing vectors, see \cite{candes2011simple,rao_universal_2012}.
\end{remark}

\begin{remark} 
The authors in \cite{jagatap_fast_2017} proposed an algorithm which achieves exact reconstruction up to global sign from $O(\frac{s^2}{B}\log(n))$. This is worse than our scaling but we stress the fact that our guarantee is on the minimizers of the optimization problem \eqref{eq:eqminregulPRexact}, while theirs is on an actual iterative reconstruction algorithm. As we discussed in the introduction, whether a tractable algorithm exists with provable exact phase retrieval guarantees from Gaussian sensing vectors under the same complexity bounds as ours is still an open question that we leave to a future work. 
\end{remark}

\subsection{Recovery bounds for frame analysis-type regularizers}\label{sec:recoverynoiselessgaussianana}
Analysis-type priors build upon the assumption that the signal of interest $\avx$ is of low complexity (\eg sparse or block sparse) after being transformed by a so-called analysis operator. Given $D \in \bbR^{n \times p}$, we consider analysis-type regularizers of the form
\begin{equation}\label{eq:Rana}
R(x) = \gamma_{\calC}(D^\top x) ,
\end{equation}
where $\gamma_{\calC}$ is a strong gauge (see Definition~\ref{def:Rstrong} and the discussion just after). Since $\gamma_\calC$ has a full domain, we have
\begin{equation}\label{eq:decompstrongana}
\partial R(x) = D \partial\gamma_{\calC}(D^\top x) = D\enscond{v \in \bbR^p}
{
  v_{\T_{D^\top x}} = \e{D^\top x}
  \qandq
  \sigma_{\calC}(v_{\S_{D^\top x}}) \leq 1
} ,
\end{equation}
where $\e{D^\top x}$ and $\T_{D^\top x}$ are the model parameters of $\gamma_\calC$ at $D^\top x$.

In this section, we will assume that $D$ is a Parseval tight frame of $\bbR^n$, meaning that $DD^\top=\Id_n$, and thus $D$ is surjective. Many popular (sparsifying) transforms in signal and image processing are Parseval tight frames (e.g. wavelets, curvelets, or concatenation of orthonormal bases; see \cite{starckfadili2015sparse}).

We can now state the following analysis-type prior version of Lemma~\ref{lem:gwidthbndgauge}.
\begin{lemma}\label{lem:gwidthbndgaugeana}
Let $R$ be of the form \eqref{eq:Rana}, where $\gamma_\calC$ is a strong gauge and $D$ is a Parseval tight frame. Let $W = D^\top Z$ where $Z \sim \calN(0,\Id_n)$. Then for any $x \in \bbR^n \setminus \ens{0}$
\begin{equation}\label{eq:gwidthbndgaugeana}
w\Ppa{\calD_R(x) \cap \bbS^{n-1}}^2 \leq \esp{\sigma_{\calC}(W_{\S_{D^\top x}})^2}\normm{\e{D^\top x}}^2 + \dim\Ppa{\T_{D^\top x}} .
\end{equation}
\end{lemma}
The proof bears some similarities with that of Lemma~\ref{lem:gwidthbndgauge}, but handling the presence of $D$ necessitates new arguments.
\begin{proof}
Since $D$ is surjective and $\gamma_\calC$ is a strong gauge, we have $\Argmin(R) = \ens{0}$. We can then argue as in the proof of Lemma~\ref{lem:gwidthbndgauge} using \cite[Proposition~3.6]{chandrasekaran_convex_2012} and \cite[Theorem~23.7]{rockafellar_convex_1970} to get that for any $x \neq 0$
\[
w\Ppa{\calD_R(x) \cap \bbS^{n-1}}^2 \leq \inf_{t \geq 0} \esp{\dist\Ppa{Z,t\partial R(x)}^2} \leq \esp{\dist\Ppa{Z,\tilde{t}D\partial \gamma_{\calC}(D^\top x)}^2}
\]
for any $\tilde{t} \geq 0$, where we have also used \eqref{eq:decompstrongana}.
Let us pick $v \in \bbR^p$ such that $v_{\S_{D^\top x}} = W_{\S_{D^\top x}}$ and $v_{\T_{D^\top x}} = \sigma_{\calC}(W_{\S_{D^\top x}})\e{D^\top x}$. Obviously, $v \in \sigma_{\calC}(W_{\S_{D^\top x}})\partial \gamma_{\calC}(x)$. We then have
\begin{align*}
w\Ppa{\calD_R(x) \cap \bbS^{n-1}}^2
&\leq \esp{\normm{Z-Dv}^2} \\
\text{\scriptsize{($D$ is a Parseval tight frame)}} &= \esp{\normm{DD^\top Z-Dv}^2} \\
\text{\scriptsize{($\normm{D}=1$)}} &\leq \esp{\normm{W-v}^2} \\
&= \esp{\normm{(W_{\T_{D^\top x}}-v_{\T_{D^\top x}}) + (W_{\S_{D^\top x}}-v_{\S_{D^\top x}})}^2} \\
&= \esp{\normm{W_{\T_{D^\top x}}-\sigma_{\calC}(W_{\S_{D^\top x}})\e{{D^\top x}}}^2} \\
&= \esp{\sigma_{\calC}(W_{\S_{D^\top x}})^2}\normm{\e{{D^\top x}}}^2 + \esp{\normm{W_{\T_{D^\top x}}}^2} .
\end{align*}
In the last equality, we used that $W_{\T_{D^\top x}}$ and $W_{\S_{D^\top x}}$ are zero-mean and uncorrelated since $\T_{D^\top x}$ and $\S_{D^\top x}$ are orthogonal, hence independent as they are Gaussian. Let $\varsigma(M)$ be the decreasing sequence of singular values of $M$. We have
\begin{align*}
\esp{\normm{W_{\T_{D^\top x}}}^2}
&= \esp{\tr\Ppa{\proj_{\T_{D^\top x}}D^\top Z Z^\top D\proj_{\T_{D^\top x}}}} \\
&= \tr\Ppa{\proj_{\T_{D^\top x}}D^\top D\proj_{\T_{D^\top x}}} \\
&= \tr\Ppa{\proj_{\T_{D^\top x}}D^\top D} \\
\text{\scriptsize{(von Neumann's trace inequality \cite{vonNeumann1962some})}} &\leq \pscal{\varsigma\bPa{\proj_{\T_{D^\top x}}},\varsigma\pa{D^\top D}} \\
\text{\scriptsize{(H\"older's inequality)}} &\leq \norm{\varsigma\bPa{\proj_{\T_{D^\top x}}}}{1}\normm{D}^2 \\
\text{\scriptsize{($\normm{D}=1$ and standard properties of orthogonal projectors on subspaces)}} &= \dim\Ppa{\T_{D^\top x}} .
\end{align*}
\end{proof}

The remaining step to get a sample complexity bound via Corollary~\ref{cor:exactrecoverygwidth} is to compute the expectation in the upper-bound \eqref{eq:gwidthbndgaugeana}. We directly consider the case where $\gamma_\calC$ is the group Lasso and the Lasso is a special case by taking blocks/groups of size $1$. 

\paragraph{Frame analysis group Lasso}
In this case, $\gamma_\calC = \norm{\cdot}{1,2}$ (see \eqref{eq:Rglasso}), and thus $\T_{D^\top x}$, $\e{D^\top x}$ and $\sigma_{\calC}$ are given by \eqref{eq:exglasso} and \eqref{eq:exglasso-g} replacing $x$ by $D^\top x$.
Thus if $\avx$ is $s$-block sparse in the dictionary $D^\top$, \ie $\abs{\bsupp(D^\top \avx)}=s$, then $\dim(\T_{D^\top \avx}) = sB$ and $\normm{\e{D^\top \avx}}^2 = s$. It remains to compute \linebreak $\esp{\max_{i \in \bsupp(\avx)^c} \normm{W[b_i]}^2}$. Note that some care has to be taken as the entries of $W$ are zero-mean Gaussian but are not independent (except in the obvious case where $D$ is orthonormal). For $t > 0$, we have
\begin{align*}
\esp{\max_{i \in \bsupp(\avx)^c} \normm{W[b_i]}^2}
&= \frac{\log\Ppa{\exp\Ppa{t\esp{\max_{i \in \bsupp(\avx)^c} \normm{W[b_i]}^2}}}}{t} \\
\text{\scriptsize{(Jensen's inequality)}} &\leq \frac{\log\Ppa{\esp{\exp\Ppa{t\max_{i \in \bsupp(\avx)^c} \normm{W[b_i]}^2}}}}{t} \\
\text{\scriptsize{(Monotonicity of the exponential)}} &= \frac{\log\Ppa{\esp{\max_{i \in \bsupp(\avx)^c} \exp\Ppa{t\normm{W[b_i]}^2}}}}{t} \\
\text{\scriptsize{(Bound the max by the sum)}} &\leq \frac{\log\Ppa{\sum_{i \in \bsupp(\avx)^c}\esp{\exp\Ppa{t \normm{W[b_i]}^2}}}}{t} .
\end{align*}
Now, for any block $b$, we have 
\[
\normm{W[b]}^2 = Z^\top D_b D_b^\top Z .
\]
The matrix $D_b D_b^\top$ is symmetric semidefinite positive and $\rank(D_b D_b^\top) \leq \min(B,n) \leq B$ provided that $B \leq n$ (in practice, we even have $B \ll n$). Moreover $\lambda_{\max}(D_b D_b^\top) \leq 1$. $D_b D_b$ can be diagonalized as $D_b D_b^\top = U \Lambda U^\top$, where $U$ is orthogonal and $\Lambda$ 
is a diagonal matrix with the eigenvalues $1 \geq \lambda_1 \geq \cdots \geq \lambda_n \geq 0$ of $D_b D_b^\top$ in its diagonal. Observe that $\lambda_i=0$ for $i \geq B+1$. Thus
\[
\normm{W[b]}^2 = \sum_{i=1}^B \lambda_i Y[i]^2 \leq \sum_{i=1}^B Y[i]^2 \qwhereq Y = U^\top Z .
\]
By the rotational invariance of the standard multivariate normal distribution, the distribution of $Y$ is the same as that of $Z$, that is, $Y \sim \calN(0,\Id_n)$. In turn, $\sum_{i=1}^B Y[i]^2$ is a $\chi$-squared random variable with $B$ degrees of freedom. Therefore, for $t \in ]0,1/2[$, we get
\begin{align*}
\esp{\max_{i \in \bsupp(\avx)^c} \normm{W[b_i]}^2}
&\leq \frac{\log\Ppa{(L-s)\esp{\exp\Ppa{t \sum_{i=1}^B Y[i]^2}}}}{t} \\
&= \frac{\log(L-s) - \frac{B}{2}\log(1-2t)}{t} ,
\end{align*}
where we used the logarithm of the moment-generating of a $\chi$-squared random variable in the last inequality. Minimizing wrt to $t$ we get that
\begin{align*}
\esp{\max_{i \in \bsupp(\avx)^c} \normm{W[b_i]}^2}
&\leq 2\log(L-s) + 2\sqrt{B\log(L-s)} + B \leq \Ppa{\sqrt{2\log(L-s)} + \sqrt{B}}^2 .  
\end{align*}
Inserting the above in Lemma~\ref{lem:gwidthbndgaugeana}, and using Corollary~\ref{cor:exactrecoverygwidth} together with Jensen's inequality, we get the following.
\begin{proposition}\label{prop:exactrecoverygl12ana}
Let $\avx$ be such that $D^\top\avx$ is an $s$-block sparse vector where the size of the blocks $B$ verifies $B \leq n$. Let $A: \bbR^n \to \bbR^m$ be a Gaussian map with \iid $\calN(0,1/m)$ entries such that 
\[
m \geq \frac{64(1+t)(\nu+2)^2}{\nu^4} s\Ppa{\Ppa{\sqrt{2\log(L-s)} + \sqrt{B}}^2+B} , t > 0 .
\]
Then with probability at least $1-3e^{-\frac{t\nu^2m}{8}}$, the recovery of $\avx$ (up to a global sign) is exact by solving \eqref{eq:eqminregulPRexact} with $R=\norm{D^\top \cdot}{1,2}$.
\end{proposition}

\begin{remark}$\ $
\begin{itemize}
    \item Consequently, it is sufficient to have $m \gtrsim s\log(p/B-s) + sB$ to ensure the exact recovery of a vector whose coefficients are $s$-sparse in a tight frame $D$, from $m$ phaseless measurements of a Gaussian map $A$. We are not aware of any such result in the phase recovery literature. 
    
    \item The analysis sparse case is directly covered by taking $B=1$ and $L=n$.
    
    \item Observe also that the sample complexity bound we get is nearly (up to constants) the same as for exact recovery from (linear) compressed sensing with Gaussian measurements \cite{candes2011simple}. 
    
\end{itemize}
\end{remark}

\subsection{Recovery bounds for total variation}\label{sec:recoverynoiselessgaussiantv}
Total variation (TV) corresponds to the case where the analysis operator $D^\top$ in \eqref{eq:Rana} is the (discrete) gradient $\nabla$ and $\gamma_\calC=\norm{\cdot}{1}$. In the 1D case, TV regularization reads
\[
R(x) = \norm{\nabla x}{1}, \qwhereq \nabla x[i] = x[i+1] - x[i], \qforq i = 1, 2, \ldots ,n-1 . 
\]
$R$ promotes signals $x$ whose gradient is sparse, $\abs{\supp(\nabla x)} \leq s$, or in other words, signals that are piecewise constant with at most $s$ jumps.

Bounding the Gaussian width of the descent cone of $R$ in this case is very challenging as $\nabla$ has a non-trivial kernel, and thus does not fit within the setting of the previous section. 
However, if the jumps of an $s$-gradient sparse signal $x$ are well separated, \cite{Genzel21} proposed a non-trivial construction of the dual vector to compute the Gaussian width of the descent cone of TV in 1D. More precisely, assume that there exists $\Delta > 0$ such that
\[
\min_{i \in \bbrac{s+1}} \frac{|k_i - k_{i-1}]}{n} \geq \frac{\Delta}{s+1}, 
\]
where $\supp(\nabla x) = \ens{i_1, \ldots, i_s}$ with $0 = i_0 < i_1 < \ldots < i_s < i_{s+1} = n$. It was shown in \cite[Theorem~2.10]{Genzel21} that if $\Delta \geq 8s/n$, then
\[
w\Ppa{\calD_{\norm{\nabla \cdot}{1}}(x) \cap \bbS^{n-1}}^2 \leq \frac{C}{\Delta} s \log(n)^2 , 
\]
for some numerical constant $C > 0$.

We are then able to state the following result.
\begin{proposition}\label{prop:exactrecoverygtv}
Let $\avx$ be such that its  gradient is sparse, $\abs{\supp(\nabla x)} \leq s$ such that its separation constant $\Delta$ verifies $\Delta \geq 8s/n$. Let $A: \bbR^n \to \bbR^m$ be a Gaussian map with \iid $\calN(0,1/m)$ entries such that  
\[
m \gtrsim \frac{1}{\Delta} s\log(n)^2 .
\]
Then with probability at least $1-3e^{-\frac{\nu^2m}{16}}$, the recovery of $\avx$ (up to a global sign) is exact by solving \eqref{eq:eqminregulPRexact} with $R=\norm{\nabla \cdot}{1}$.
\end{proposition}
\begin{remark}
As mentioned before, finding complexity bounds for TV minimization is quite challenging even in the compressed sensing literature. In this setting, \cite{needell_near-optimal_2013,needell_stable_2013, cai_guarantees_2015} showed, for two or higher dimensions signals, robust and stable recovery when $A$ is Gaussian and composed with orthonormal Haar wavelet transform. The complexity in this case is of order $m\geq s\text{PolyLog}(n,s)$. The success of this approach relies on establishing a connection between the compressibility of Haar wavelet representations and the bounded variation of a function and this does not hold in one dimension. We think that it might be possible to extend this result to the case of phase retrieval and we leave this as future work. 
\end{remark}


\section{Stable Recovery: Constrained Problem}\label{sec:robustconstrained}


When we have access only to inaccurate noisy measurements as in \eqref{eq:GeneralPR}, a natural formulation is one in which the equality constraint in \eqref{eq:eqminregulPRexact} is relaxed to an inequality leading to
\begin{equation}\tag{$\scrP_{y,\rho}$}\label{eq:eqminregulPRconst}
\inf_{x\in\bbR^n} R(x) \qsubjq \normm{y-|Ax|^2} \leq \rho ,
\end{equation}
where $\rho$ is an upper bound on the size of the noise $\epsi$. In the inverse problems literature, this formulation is known as the residual method or Mozorov regularization. In the following, we denote $\overbar{\calF}_{y,\rho} \eqdef \enscond{w \in \bbR^n}{\normm{y-|w|^2} \leq \rho}$. We obviously have $\overbar{\calF}_{\avy,0} = \overbar{\calF}$. We also use the shorthand notation $\calS_{y,\rho}$ for the set of minimizers of \eqref{eq:eqminregulPRconst}. \\

We start by showing that \eqref{eq:eqminregulPRconst} has minimizers. This result does not require convexity of $R$.
\begin{proposition}\label{prop:existenceconst}
Let $R: \bbR^n \to \barR$ be a proper and lsc function. Assume that $A(\dom(R)) \cap \overbar{\calF}_{y,\rho} \neq \emptyset$, and that assumptions \ref{H:Rbnd}-\ref{H:RinfA} of Proposition~\ref{prop:existencenoiseless} hold. Then $\calS_{y,\rho}$ is a non-empty compact set. 
\end{proposition}

\begin{proof}
The proof is similar to that of Proposition~\ref{prop:existencenoiseless} replacing $\overbar{\calF}$ by $\overbar{\calF}_{y,\rho}$, and using compactness of the latter.
\end{proof}

We are now ready to state our (deterministic) stability result.

\begin{theorem}\label{thm:stablerecoveryconstdet}
Consider the noisy phaseless measurements in \eqref{eq:GeneralPR} where $\normm{\epsi} \leq \rho$. Assume that $\calS_{y,\rho} \neq \emptyset$ and $R$ verifies \ref{H:Reven}. Then, for any $\xsol_{y,\rho} \in \calS_{y,\rho}$, we have
\[
\dist(\xsol_{y,\rho},\avX) \leq \frac{2\rho}{s_{\min}} ,
\]
where 
\[
s_{\min} \eqdef \inf\enscondlr{\min_{I \subset \bbrac{m}, |I| \geq m/2} \normm{A^{I} z}}{z \in \calD_R(\avx) \cap \bbS^{n-1}} > 0 .
\]
\end{theorem}

\begin{proof}
The proof has a flavour of the reasoning in the proof of Theorem~\ref{thm:exactrecoverydet}. Let $I \subset \bbrac{m}$ such that $\pscal{a_r,\xsol_{y,\rho}} = \pscal{a_r,\avx}$ for all $r \in I$, and $I^c$ its complement where the inner products have opposite signs. Thus either $|I| \geq m/2$ or $|I^c| \geq m/2$. Assume that $|I| \geq m/2$. Then
\begin{align*}
\normm{|A\xsol_{y,\rho}| - |A\avx|}^2 
&= \normm{|A^I\xsol_{y,\rho}| - |A^I\avx|}^2 + \normm{|A^{I^c}\xsol_{y,\rho}| - |A^{I^c}\avx|}^2 \\
&\geq \normm{A^I\xsol_{y,\rho} - A^I\avx}^2 .
\end{align*}
Recall that $\avx \in \overbar{\calF}_{y,\rho}$ by assumption on the noise. Thus $R(\xsol_{y,\rho}) \leq R(\avx)$ and in turn $\xsol_{y,\rho} - \avx \in \calD_R(\avx)$. Therefore,
\begin{align*}
\normm{|A\xsol_{y,\rho}| - |A\avx|}^2 
\geq \normm{A^I(\xsol_{y,\rho} - \avx)}^2 
\geq s_{\min}^2\normm{\xsol_{y,\rho} - \avx}^2 .
\end{align*}
For the case where $|I^c| \geq m/2$, we argue similarly to infer that
\begin{align*}
\normm{|A\xsol_{y,\rho}| - |A\avx|}^2 
\geq \normm{A^{I^c}(\xsol_{y,\rho} + \avx)}^2
\geq s_{\min}^2\normm{\xsol_{y,\rho} + \avx}^2 .
\end{align*}
Overall, we have
\begin{align*}
\dist(\xsol_{y,\rho},\avX) \leq \frac{\normm{|A\xsol_{y,\rho}| - |A\avx|}}{s_{\min}} \leq \frac{\normm{y-|A\xsol_{y,\rho}|} + \normm{\epsi}}{s_{\min}} \leq \frac{2\rho}{s_{\min}} .
\end{align*}
\end{proof}

When $A$ is a standard Gaussian map, we obtain the following general error bound.
\begin{proposition}\label{prop:stablerecoveryconstgaussian}
Consider the noisy phaseless measurements in \eqref{eq:GeneralPR} where $\normm{\epsi} \leq \rho$. Suppose that \ref{H:Reven} holds. Let $\nu$ be as defined in Lemma~\ref{lem:sigAI} and $A$ be a Gaussian map with \iid $\calN(0,1/m)$ entries such that
\[
m \geq \frac{64(1+t)(\nu+2)^2}{\nu^4} w\Ppa{\calD_R(\avx) \cap \bbS^{n-1}}^2 
\]
for some $t > 0$. Then with probability at least $1-3e^{-\frac{t\nu^2m}{8}}$, the following holds:
\[
\dist(\xsol_{y,\rho},\avX) \leq \frac{4\rho}{\nu(1-1/\sqrt{2})} \qforanyq \xsol_{y,\rho} \in \calS_{y,\rho} .
\]
\end{proposition}

\begin{proof}
Taking $\eps=\frac{\nu}{\sqrt{2}(2+\nu)}$ in \eqref{eq:sminconcent} as devised in Corollary~\ref{cor:exactrecoverygwidth}, we have
\[
s_{\min} \geq \nu/2(1-1/\sqrt{2}) > 0 
\] 
with probability at least $1-3e^{-\frac{t\nu^2m}{8}}$ under the bound on $m$. Combining this with Theorem~\ref{thm:stablerecoveryconstdet}, we conclude.
\end{proof}

In \cite{gao_stable_2016}, the authors studied the stability of $\ell_1-$norm phase retrieval against noise and showed that for $m \gtrsim s\log(n/s)$, any $s-$sparse vector can be stably recovered from measurement maps $A$ that satisfy the strong-RIP property. Our stability result here is RIP-less. Moreover it goes far beyond the $\ell_1-$norm. 
       
For Gaussian measurements, Proposition~\ref{prop:stablerecoveryconstgaussian} gives a sample complexity bound that depends on the Gaussian width of the descent cone. We can easily instantiate the last result for the regularizers studied in Section~\ref{sec:recoverynoiselessgaussiandecomp}, \ref{sec:recoverynoiselessgaussianana} and \ref{sec:recoverynoiselessgaussiantv}, which in turn will give sample complexity bounds for the error bound of Theorem~\ref{thm:exactrecoverygaussian} to hold. We refrain from doing this for the sake of brevity and the straightforward details are left to the reader.

Let us finally notice that despite the nice stability properties enjoyed by \eqref{eq:eqminregulPRconst}, this problem seems challenging to solve numerically. Indeed, although $R$ is convex, the constraint in \eqref{eq:eqminregulPRconst} is highly non-convex, and it is an open problem to design an efficient algorithmic scheme to solve it. On the other hand, as stated in the introduction, \eqref{eq:eqminregulPRpen} is amenable to the efficient Bregman Proximal Gradient algorithmic scheme proposed \cite{bolte_first_2017}. This is the reason we now turn our attention to \eqref{eq:eqminregulPRpen}.
        


\section{Stable Recovery: Penalized Problem}\label{sec:robustpenalized}

We now turn to study the noise-aware problem \eqref{eq:eqminregulPRpen}. In particular, the following questions will be of most interest to us:
\begin{enumerate}[label=(Q.\arabic*)]
	\item Convergence: how to ensure that for $\epsi \to 0$, the set of regularized solutions
	converges to either $\avx$ or $-\avx$ ? \label{Q:conv}
	
	\item Convergence rates: at which rate (in term of noise) the above convergence takes place ? \label{Q:rate}
\end{enumerate}

\medskip

We will answer these two questions in the rest of the paper. In a nutshell, we will show that we indeed have convergence as the noise vanishes, and stability only occurs locally, \ie for small enough noise, with an error bound that scales linearly with the noise level.

As we will see, studying the stability of \eqref{eq:eqminregulPRpen} is more involved than for \eqref{eq:eqminregulPRconst}. One of the main difficulties, which was also highlighted for linear inverse problems (see \cite{vaiter2015low}), is that a minimizer of \eqref{eq:eqminregulPRpen} is not anymore in the descent cone of $R$ at $\avx$.

In this section, we set 
\[
\calS_{y,\lambda} \eqdef \Argmin_{x \in \bbR^n} F_{y,\lambda}(x) ,
\]
where we recall the objective $F_{y,\lambda}$ from \eqref{eq:eqminregulPRpen}.

We begin by providing conditions for the existence of minimizers. Again, this does not need convexity of $R$.
\begin{proposition}\label{prop:existencepen}
	Let $R: \bbR^n \to \barR$ be a proper and lsc function. Assume that assumptions \ref{H:Rbnd}-\ref{H:RinfA} of Proposition~\ref{prop:existencenoiseless} hold. Then for any $\lambda > 0$ and $y \in \bbR^m$, problem \eqref{eq:eqminregulPRpen} has a non-empty compact set of minimizers. 
\end{proposition}

\begin{proof}
	The proof is similar to that of Proposition~\ref{prop:existencenoiseless} replacing $\iota_{\overbar{\calF}}$ by $\normm{y-|\cdot|^2}^2/2$, and the latter turns out to be a smooth and coercive function.
\end{proof}

\subsection{Convergence}
We start by answering \ref{Q:conv} above and proving the following convergence result for any minimizer $\xsol_{y,\lambda}$ of \eqref{eq:eqminregulPRpen}. This can be seen as a $\Gamma$-convergence result of the objective in \eqref{eq:eqminregulPRpen} to that of \eqref{eq:eqminregulPRexact}.
\begin{theorem}\label{thm:stablerecoverypendet}
	Consider the noisy phaseless measurements in \eqref{eq:GeneralPR}. Let $\sigma \eqdef \normm{\epsi}$. Assume that \ref{H:Reven}, \ref{H:injdescent} and assumptions \ref{H:Rbnd}-\ref{H:RinfA} of Proposition~\ref{prop:existencenoiseless} hold. Suppose also that
	\[
	\lambda \to 0 \qandq \sigma^2/\lambda \to 0, \qasq \sigma \to 0 .
	\]
	Then,
	\[
	|A \xsol_{y,\lambda}| \to |A \avx|, \quad R(\xsol_{y,\lambda}) \to R(\avx) \qandq \dist\Ppa{\xsol_{y,\lambda},\avX} \to 0 \qasq \sigma \to 0 .
	\]
\end{theorem}

\begin{proof}
	Let $y_k = \avy + \epsi_k$, $\sigma_k = \normm{\epsi_k}$ with $\sigma_k \to 0$ as $k \to +\infty$. Observe that for any $y_k$ and $\lambda_k > 0$ $\calS_{y_k,\lambda_k}$ is a non-empty compact set thanks to Proposition~\ref{prop:existencepen}. Let $\xsol_k \in \calS_{y_k,\lambda_k}$. We have by optimality that
	\begin{align*}
		\normm{y_k-|A \xsol_k|^2}^2 + \lambda_k R(\xsol_k)
		\leq \normm{y_k-\avy}^2 + \lambda_k R(\avx)
		= \sigma_k^2 + \lambda_k R(\avx) .
	\end{align*}
	Thus
	\begin{equation*}
		\begin{gathered}
			\normm{y_k-|A \xsol_k|^2}^2 \leq \lambda_k\Ppa{\sigma_k^2/\lambda_k + R(\avx)} \\
			\qandq \\
			R(\xsol_{y_k,\lambda_k}) \leq \sigma_k^2/\lambda_k + R(\avx) .
		\end{gathered}
	\end{equation*}
	In turn, 
	\[
	\normm{|A \xsol_k|^2 - \avy}^2 \leq 2\Ppa{\normm{y_k-|A \xsol_k|^2}^2 + \sigma_k^2} \leq 2\Ppa{\lambda_k\Ppa{\sigma_k^2/\lambda_k + R(\avx)} + \sigma_k^2} .
	\]
	Since the right hand side of this inequality goes to $0$ as $k \to +\infty$, we deduce that 
	\begin{equation}\label{eq:limAxk}
	\lim_{k \to +\infty} |A \xsol_k|^2 = \avy .
	\end{equation}
	Moreover,
	\begin{equation}\label{eq:limRxk}
	\limsup_{k \to +\infty} R(\xsol_k) \leq R(\avx) .
	\end{equation}
	We therefore obtain
	\[
	\limsup_{k \to +\infty} F_{\avy,1}(\xsol_k)  = \limsup_{k \to +\infty} \Ppa{\normm{|A \xsol_k|^2 - \avy}^2 + R(\xsol_k)} \leq R(\avx) .
	\]
	This means that there exists $k_0 \in \bbN$ such that $\Ppa{\xsol_k}_{k \geq k_0}$ belongs to the sublevel set of
	$F_{\avy,1}$ at $2R(\avx)$, that we denote $\calC_F$. Since $F_{\avy,1}$ is lsc and coercive under our assumptions, its sublevel sets are compact and so is $\calC_F$. In turn, $\Ppa{\xsol_k}_{k \geq k_0}$ lives on the compact set $\calC_F$. The sequence thus possesses a convergent subsequence and every accumulation point lies $\calC_F$. Let $\Pa{\xsol_{k_j}}_{j \in \bbN}$ be a convergent subsequence, say $\xsol_{k_j} \to x^*$. We have $|A\xsol_{k_j}|^2 \to |Ax^*|^2$, and in view of \eqref{eq:limAxk}, we obtain 
	\begin{equation}\label{eq:x*feasible}
	|Ax^*|^2 = \avy . 
	\end{equation}
	Moreover, by lower-semicontinuity of $R$ and \eqref{eq:limRxk}, 
	\begin{equation}\label{eq:x*minimal}
	R(x^*) \leq \liminf_{j \to +\infty} R(x^\star_{k_j}) \leq \limsup_{j \to +\infty} R(\xsol_{k_j}) \leq R(\avx) .
	\end{equation}
	Invoking Theorem~\ref{thm:exactrecoverydet} (which holds under our assumptions), \eqref{eq:x*feasible} and  \eqref{eq:x*minimal} mean that $x^* \in \calS_{\avy,0} = \avX$. \eqref{eq:x*minimal} also tells us that $R(\avx) = R(x^*)$, and thus we have $R(x^\star_{k_j}) \to R(\avx)$. Since this holds for any convergent subsequence, we conclude.
\end{proof}

Thanks to Corollary~\ref{cor:exactrecoverygwidth}, we know that \ref{H:injdescent} holds with high probability when $A$ is a Gaussian map provided that $m$ is large enough. Combining this with Theorem~\ref{thm:stablerecoverypendet}, we get the following asymptotic robustness result.
\begin{proposition}\label{prop:stablerecoverypengaussian}
Consider the noisy phaseless measurements in \eqref{eq:GeneralPR} and let $\sigma \eqdef \normm{\epsi}$. Assume that $R$ fullfills \ref{H:Reven} and assumptions \ref{H:Rbnd}-\ref{H:RinfA} of Proposition~\ref{prop:existencenoiseless} hold. Suppose also that
	\[
	\lambda \to 0 \qandq \sigma^2/\lambda \to 0, \qasq \sigma \to 0 .
	\]
	Let $\nu$ be as defined in Lemma~\ref{lem:sigAI}, and $A$ be a Gaussian map with \iid $\calN(0,1/m)$ entries such that
	\[
	m \geq \frac{64(1+t)(\nu+2)^2}{\nu^4} w\Ppa{\calD_R(\avx) \cap \bbS^{n-1}}^2 
	\]
	for some $t > 0$. Then with probability at least $1-3e^{-\frac{t\nu^2m}{8}}$,
	\[
	\dist\Ppa{\xsol_{y,\lambda},\avX} \to 0 \qasq \sigma \to 0 .
	\]
\end{proposition}
This result can be specialized with the corresponding sample complexity bounds for each of the regularizers considered in Section~\ref{sec:recoverynoiselessgaussiandecomp}, \ref{sec:recoverynoiselessgaussianana} and \ref{sec:recoverynoiselessgaussiantv}. We omit again the details which are left to the reader.

\subsection{Convergence rate}
We now turn to answering \ref{Q:rate} by quantifying the rate at which convergence of Theorem~\ref{thm:stablerecoverypendet} occurs. This will be possible under more stringent conditions. For instance, we will require the noise to be small enough to that the rate is actually local as it is already known for phase retrieval in the un-regularized case. We moreover need a non-degeneracy condition and a restricted injectivity conditions which are standard in inverse problems; see Remark~\ref{rem:ratestablepen} for a detailed discussion. 

To lighten notation, let us denote 
\[
\BAx \eqdef \diag{A\avx} A.
\] 
The rationale behind this operator is that $\BAx \avx = |A\avx|^2$, and thus $\BAx$ will appear naturally from a linearization of the forward model. Indeed, $\BAx$ is nothing but the Jacobian of the non-linear mapping $x \in \bbR^n \mapsto |A x|^2/2$ at $\avx$.

Although the following result can be stated for general symmetric convex regularizers $R$, to avoid additional technicalities and make the presentation simpler, we will restrict our attention to the case of analysis-type symmetric strong gauges which will be sufficient for our purposes. More precisely, $R$ will be of the form \eqref{eq:Rana}, where $D$ is a Parseval tight frame and $\gamma_\calC$ is a symmetric strong gauge. We recall the definition, notations, and properties of Section~\ref{sec:recoverynoiselessgaussiandecomp}.
\begin{theorem}\label{thm:ratestablependet}
Consider the noisy phaseless measurements in \eqref{eq:GeneralPR}. Let $\sigma \eqdef \normm{\epsi}$. Assume that $R$ is as in \eqref{eq:Rana}, where $D$ is a Parseval tight frame and $\gamma_\calC$ is a symmetric strong gauge, and let $\lambda = c\sigma$, for some $c > 0$. Then, the following holds.
\begin{enumerate}[label=(\roman*)]
\item If 
\begin{equation}\label{eq:sc}\tag{SC}
\exists q \in \bbR^m \quad s.t. \quad {\BAx}^\top q \in \partial R(\avx) ,
\end{equation}
then for any minimizer $\xsol_{y,\lambda} \in \calS_{y,\lambda}$,
\begin{equation}\label{eq:boundpredictionbregman}
\begin{gathered}
\normm{|A\xsol_{y,\lambda}|^2 - |A\avx|^2} 
\leq \frac{\normm{A}^2}{2}\normm{\xsol_{y,\lambda}-\avx}^2 + \Ppa{2+c\normm{q}}\sigma \\ \qandq \\
\breg{R}{v}{\xsol_{y,\lambda}}{\avx}
\leq \frac{\normm{q}}{2}\normm{A}^2\normm{\xsol_{y,\lambda}-\avx}^2 + \frac{\Ppa{2+\frac{c}{2}\normm{q}}^2}{2c}\sigma .
\end{gathered}
\end{equation}

\item If
\begin{equation}\label{eq:nondeg}\tag{NDSC}
\exists q \in \bbR^m \quad s.t. \quad {\BAx}^\top q \in \ri(\partial R(\avx))
\end{equation}
and 
\begin{equation}\label{eq:restinj}\tag{RI}
\ker(\BAx) \cap \Im(D_{\T_{D^\top\avx}}) = \ens{0} .
\end{equation}
then for $\sigma$ small enough and any minimizer $\xsol_{y,\lambda} \in \calS_{y,\lambda}$, we have
\[
\dist\Ppa{\xsol_{y,\lambda},\avX} \leq C \sigma ,
\]
where $C>0$ is a constant which depends in particular on $A$, $\T_{D^\top\avx}$, $c$ and $q$.
\end{enumerate}
\end{theorem}

A few remarks are in order before we proceed with the proof.
\begin{remark}\label{rem:ratestablepen}{\ }
	\begin{itemize}
		\item since $R$ verifies all assumptions of Theorem~\ref{thm:stablerecoverypendet}, we have that $\calS_{y,\lambda}$ are bounded uniformly in $(y,\lambda)$, and it follows from \eqref{eq:boundpredictionbregman} that the error $\normm{|A\xsol_{y,\lambda}|^2 - |A\avx|^2}$ is global and scales as $O(\max(\sigma^{1/2},\sigma))$.
		
		\item The error bound of Theorem~\ref{thm:ratestablependet} tells us that for small noise, the distance of any minimizer of \eqref{eq:eqminregulPRpen} to $\avX$ is within a factor of the noise level, which justifies the terminology "linear convergence rate" known in the inverse problem literature. 
		
		\item Non-degenerate Source Condition: the condition \eqref{eq:nondeg} is a strengthened or non-degenerate version of \eqref{eq:sc} well-known as the "source condition" or "range condition" in the literature of inverse problems; see \cite{scherzer2009} for a general overview of this condition and its implications. In this case, $v = \BAx^\top q$ is called a non-degenerate "dual" certificate\footnote{Strictly speaking, the terminology "dual" may seem awkward because of non-convexity of the phase retrieval problem \eqref{eq:eqminregulPRpen} though it is weakly convex.}; see \cite{vaiter2015low} for a detailed discussion in the case of linear inverse problems.
		
		\item Restricted Injectivity: the condition \eqref{eq:restinj} is only favorable when $\gamma_\calC$ is non-smooth at $D^\top\avx$, hence the intuition that $\gamma_\calC$ (hence $R$) promotes low-dimensional vectors. Indeed, the higher the degree of non-smoothness, the lower the dimension of the subspace $\T_{D^\top\avx}$, and hence the less number of measurements is needed for \eqref{eq:restinj} to hold. From the calculus rules in \cite[Proposition~10(i)-(ii)]{vaiterimaiai13}, the model subspace of the regularizer $R$ at $\avx$ is $\Ker(D_{\S_{D^\top\avx}}^\top)$. Since $D$ is a Parseval tight frame, one can easily show that $\Ker(D_{\S_{D^\top\avx}}^\top) \subseteq \Im(D_{\T_{D^\top\avx}})$, with equality if $D$ orthonormal. Thus \eqref{eq:restinj} implies that $\BAx$ is injective on the model subspace of $R$ at $\avx$, which is a minimal requirement to ensure recovery as is known even for linear inverse problems.
		
		\item Convergence rates for regularized non-linear inverse problems were studied in \cite{scherzer2009}. However, their conditions are too stringent and do not hold for the case of phase retrieval by solving \eqref{eq:eqminregulPRpen}.
	\end{itemize}
\end{remark}

The following lemma is a key step towards establishing our error bound.
\begin{lemma}\label{lem:nonsat}
Let $R$ be as in \eqref{eq:Rana} where $D \in \bbR^{p \times n}$ and $\gamma_\calC$ is a strong gauge of $\calC$. Let $x \in \bbR^n$. Then, for any $w \in \ri(\gamma_{\calC}(D^\top x))$ and $z \in \bbR^n$
\begin{align*}
		\gamma_\calC\Ppa{\proj_{\S_{D^\top x}}D^\top(z - x)} \leq \frac{\breg{R}{Dw}{z}{x}}{1-\sigma_\calC\Ppa{w_{\S_{D^\top x}}}} = \frac{\breg{\gamma_\calC}{w}{D^\top z}{D^\top x}}{1-\sigma_\calC\Ppa{w_{\S_{D^\top x}}}} .
\end{align*}
\end{lemma}
Observe that by the decomposability property in \eqref{eq:decompstrong} the following equivalent holds:
\begin{equation}\label{eq:ridualcertstrong}
w \in \ri(\partial \gamma_\calC(D^\top x)) \iff w_{\T_{D^\top z}} = \e{D^\top x} \qandq \sigma_\calC(w_{\S_{D^\top x}}) < 1 .
\end{equation}
In plain words, the denominator in Lemma~\ref{lem:nonsat} does not vanish and the statement is not vacuous. In fact, this denominator can be viewed as a ``distance'' to degeneracy.

\begin{proof} 
	We start by noting that $\ri(\partial R(x)) = D \ri(\partial\gamma_\calC(D^\top x))$. Let $v = D w$. We have by convexity and decomposability of the subdifferential of $\gamma_\calC$ that for any pair $(u,w) \in \partial \gamma_\calC(D^\top x) \times \ri\Ppa{\partial \gamma_\calC(D^\top x)}$,
	\begin{align*}
		\breg{R}{v}{z}{x}	&= \breg{\gamma_\calC}{w}{D^\top z}{D^\top x}\\
		&\geq \breg{\gamma_\calC}{w}{D^\top z}{D^\top x} - \breg{\gamma_\calC}{u}{D^\top z}{D^\top x} \\
		&= \pscal{u-w,D^\top z-D^\top x} \\
		&= \pscal{u_{\S_{D^\top z}}-w_{\S_{D^\top z}},D^\top z-D^\top x} .
	\end{align*}
	From \cite[Theorem~1]{luu_sharp_2020}, specialized to strong gauges, we have that for any $\omega \in \bbR^p$, $\exists \tilde{u} \in \partial \gamma_\calC(D^\top x)$ such that
	\[
	\gamma_\calC(\omega_{\S_{D^\top x}}) = \pscal{\tilde{u}_{\S_{D^\top x}},\omega_{\S_{D^\top x}}} .
	\]
	Applying this with $\omega=D^\top z-D^\top x$ and taking $u = \tilde{u}$, continuing the above chain of inequalities yields
	\begin{align*}
		\breg{R}{v}{z}{x}	
		&\geq \gamma_\calC(\proj_{\S_{D^\top x}}(D^\top z-D^\top x)) - \pscal{w_{\S_{D^\top x}},\proj_{\S_{D^\top x}}(D^\top z-D^\top x)} \\
		&\geq \gamma_\calC(\proj_{\S_{D^\top x}}(D^\top z-D^\top x))\Ppa{1-\sigma_\calC(w_{\S_{D^\top x}}} ,
	\end{align*}
	where in the last inequality, we used the duality inequality which holds by polarity between $\gamma_\calC$ and $\sigma_\calC$. 
	This concludes the proof.
\end{proof}

\begin{proof}[Proof of Theorem~\ref{thm:ratestablependet}]
\noindent (i) \quad Let $\xsol_{y,\lambda} \in \calS_{y,\lambda}$ and suppose that $\avx$ is its closest point in $\avX$. We have by optimality that
\begin{align*}
\normm{y-|A \xsol_{y,\lambda}|^2}^2 + \lambda R(\xsol_{y,\lambda}) \leq \sigma^2 + \lambda R(\avx) .
\end{align*}
 the source condition \eqref{eq:nondeg}, there exists $q \in \bbR^m$ such that $v \eqdef \BAx^\top q \in \ri(\partial R(\avx))$. Convexity of $R$ then implies
\begin{align*}
\normm{y-|A \xsol_{y,\lambda}|^2}^2 + \lambda \breg{R}{v}{\xsol_{y,\lambda}}{\avx} 
&\leq \sigma^2 -\lambda\pscal{q,\BAx(\xsol_{y,\lambda}-\avx)} \\
&= \sigma^2 +\frac{\lambda}{2}\pscal{q, |A\xsol_{y,\lambda}-A\avx|^2 + (|A\xsol_{y,\lambda}|^2 - |A\avx|^2)} \\
\text{\scriptsize{(Cauchy-Schwarz inequality)}} &\leq \sigma^2 +\frac{\lambda}{2}\Ppa{\normm{q}\norm{A\xsol_{y,\lambda}-A\avx}{4}^2 + \normm{q}\normm{|A\xsol_{y,\lambda}|^2 - |A\avx|^2}} \\
\text{\scriptsize{($\norm{\cdot}{4} \leq \normm{\cdot}$)}}		&\leq \sigma^2 +\frac{\lambda}{2}\normm{q}\Ppa{\normm{A\xsol_{y,\lambda}-A\avx}^2 + \normm{|A\xsol_{y,\lambda}|^2 - |A\avx|^2}} .
\end{align*}
Strong convexity of $\normm{\cdot}^2$ implies that
\[
\normm{y-|A \xsol_{y,\lambda}|^2}^2 - \sigma^2 \geq 2\pscal{\epsi,|A\xsol_{y,\lambda}|^2 - |A\avx|^2} +  \normm{|A\xsol_{y,\lambda}|^2 - |A\avx|^2}^2 .
\]
Thus
\begin{align*}
&\normm{|A\xsol_{y,\lambda}|^2 - |A\avx|^2}^2 + \lambda \breg{R}{v}{\xsol_{y,\lambda}}{\avx} \\
&\leq -2\pscal{\epsi,|A\xsol_{y,\lambda}|^2 - |A\avx|^2} + \frac{\lambda}{2}\normm{q}\Ppa{\normm{A\xsol_{y,\lambda}-A\avx}^2 + \normm{|A\xsol_{y,\lambda}|^2 - |A\avx|^2}} \\
\text{\scriptsize{(Cauchy-Schwarz inequality)}}		&\leq \frac{\lambda}{2}\normm{q}\normm{A\xsol_{y,\lambda}-A\avx}^2 + \Ppa{2\sigma+\frac{\lambda}{2}\normm{q}}\normm{|A\xsol_{y,\lambda}|^2 - |A\avx|^2} \\
\text{\scriptsize{(Young's inequality)}}		&\leq \frac{\lambda}{2}\normm{q}\normm{A}^2\normm{\xsol_{y,\lambda}-\avx}^2 + \frac{\Ppa{2\sigma+\frac{\lambda}{2}\normm{q}}^2}{2} + \frac{1}{2}\normm{|A\xsol_{y,\lambda}|^2 - |A\avx|^2}^2 ,
\end{align*}
 therefore
\begin{align*}
\normm{|A\xsol_{y,\lambda}|^2 - |A\avx|^2}^2 + 2\lambda \breg{R}{v}{\xsol_{y,\lambda}}{\avx}
&\leq \lambda\normm{q}\normm{A}^2\normm{\xsol_{y,\lambda}-\avx}^2 + \Ppa{2\sigma+\frac{\lambda}{2}\normm{q}}^2 .
\end{align*}
Using the choice  $\lambda=c\sigma$, non-negativity of the Bregman divergence of $R$, that $\sqrt{a+b}\leq\sqrt{a}+\sqrt{b}$ and Young's inequality, we get \eqref{eq:boundpredictionbregman}.

\noindent (ii) \quad By the triangle inequality and since $D$ is a Parseval tight frame, we get
\begin{align*}
\normm{\xsol_{y,\lambda}-\avx} 
&= \normm{DD^\top(\xsol_{y,\lambda}-\avx)} \\
&\leq \normm{D\proj_{\T_{D^\top\avx}}D^\top(\xsol_{y,\lambda}-\avx)} + \normm{\proj_{\S_{D^\top\avx}}D^\top(\xsol_{y,\lambda}-\avx)} .
\end{align*}
Denote $\Vx \eqdef \Im(D_{\T_{D^\top\avx}})$. In view of \eqref{eq:restinj}, we have ${\BAx}_{\Vx}^+=({\BAx}_{\Vx}^\top{\BAx}_{\Vx})^{-1}{\BAx}_{\Vx}^\top$. Moreover, since $v \in \ri(\partial R(\avx))$, we have from \cite[Theorem~6.6]{rockafellar_convex_1970} that $v=Dw$ for some $w \in \ri(\partial\gamma_\calC(D^\top \avx))$. Therefore, observing that $D\proj_{\T_{D^\top\avx}}D^\top(\xsol_{y,\lambda}-\avx) \in \Vx$ and using Lemma~\ref{lem:nonsat}, we obtain
\begin{align*}
&\normm{\xsol_{y,\lambda}-\avx} \\
&\leq \normm{{\BAx}_{\Vx}^+{\BAx}_{\Vx}D\proj_{\T_{D^\top\avx}}D^\top(\xsol_{y,\lambda}-\avx)} +  \normm{\proj_{\S_{D^\top\avx}}D^\top(\xsol_{y,\lambda}-\avx)} \\
&= \normm{{\BAx}_{\Vx}^+{\BAx}D\proj_{\T_{D^\top\avx}}D^\top(\xsol_{y,\lambda}-\avx)} +  \normm{\proj_{\S_{D^\top\avx}}D^\top(\xsol_{y,\lambda}-\avx)} \\
&= \normm{{\BAx}_{\Vx}^+{\BAx}\Ppa{\Id_n - D\proj_{\S_{D^\top\avx}}D^\top}(\xsol_{y,\lambda}-\avx)} +  \normm{\proj_{\S_{D^\top\avx}}D^\top(\xsol_{y,\lambda}-\avx)} \\
&\leq \normm{{\BAx}_{\Vx}^+}\normm{{\BAx}(\xsol_{y,\lambda}-\avx)} + \Ppa{\anormOP{\proj_{\S_{D^\top\avx}}}{\gamma_\calC}{2}+\normm{{\BAx}_{\Vx}^+}\norm{A\avx}{\infty}\anormOP{A D_{\S_{D^\top\avx}}}{\gamma_\calC}{2}}\gamma_\calC\Ppa{\proj_{\S_{D^\top\avx}}D^\top(\xsol_{y,\lambda}-\avx)} \\
&\leq \normm{{\BAx}_{\Vx}^+}\normm{{\BAx}(\xsol_{y,\lambda}-\avx)} + \Ppa{\anormOP{\proj_{\S_{D^\top\avx}}}{\gamma_\calC}{2}+\normm{{\BAx}_{\Vx}^+}\norm{A\avx}{\infty}\anormOP{A D_{\S_{D^\top\avx}}}{\gamma_\calC}{2}}\frac{\breg{R}{v}{x}{\avx}}{1-\sigma_\calC\Ppa{w_{\S_{D^\top\avx}}}} ,
\end{align*}
where we also used coercivity of $\gamma_\calC$. Let
\begin{align*}
\alpha &= \frac{\anormOP{\proj_{\S_{D^\top\avx}}}{\gamma_\calC}{2}+\normm{{\BAx}_{\Vx}^+}\norm{A\avx}{\infty}\anormOP{AD_{\S_{D^\top\avx}}}{\gamma_\calC}{2}}{1-\sigma_\calC\Ppa{w_{\S_{D^\top\avx}}}} .
	\end{align*}
	Thus
	\begin{align*}
\normm{\xsol_{y,\lambda}-\avx}
&\leq\normm{{\BAx}_{\Vx}^+}\normm{|A\xsol_{y,\lambda}-A\avx|^2 + (|A\xsol_{y,\lambda}|^2 - |A\avx|^2)} + \alpha\breg{R}{v}{x}{\avx} \\
&\leq\normm{{\BAx}_{\Vx}^+}\Ppa{\normm{A}^2\normm{\xsol_{y,\lambda}-\avx}^2 + \normm{|A\xsol_{y,\lambda}|^2 - |A\avx|^2}} + \alpha\breg{R}{v}{x}{\avx} .
	\end{align*}
	Inserting the bounds in \eqref{eq:boundpredictionbregman} and rearranging, we get
	\begin{align*}
\normm{\xsol_{y,\lambda}-\avx} \leq b \sigma + a \normm{\xsol_{y,\lambda}-\avx}^2 ,
	\end{align*}
	where
	\begin{align*}
a = \frac{\normm{A}^2}{2}\Pa{3\normm{{\BAx}_{\Vx}^+} + \alpha\normm{q}} \qandq
b = \normm{{\BAx}_{\Vx}^+}\Ppa{2+c\normm{q}} + \alpha \frac{\Ppa{2+\frac{c}{2}\normm{q}}^2}{2c} .
	\end{align*}
	By symmetry of $R$, it can be easily seen that $\partial R(-\avx) = -\partial R(\avx)$, and thus $\T_{-D^\top\avx} = \T_{D^\top\avx}$ and $\e{-D^\top\avx}=-\e{D^\top\avx}$. Therefore, if \eqref{eq:nondeg}-\eqref{eq:restinj} hold at $\avx$ then so they do at $-\avx$ and vice versa. In turn, when $-\avx$ is the closest point to $\xsol_{y,\lambda}$, we argue similarly as above to get
	\begin{align*}
\normm{\xsol_{y,\lambda}+\avx} \leq b \sigma + a \normm{\xsol_{y,\lambda}+\avx}^2 .
	\end{align*}
	Overall, we arrive at
	\begin{align*}
\dist(\xsol_{y,\lambda},\avX) \leq b \sigma + a \dist(\xsol_{y,\lambda},\avX)^2 .
	\end{align*}
	Solving the above inequality, recalling that $\dist(\xsol_{y,\lambda},\avX)$ vanishes as $\sigma \to 0$ thanks to Theorem~\ref{thm:stablerecoverypendet}, we get that if 
	\[
	\sigma \leq 1/(4ab) ,
	\]
	Then
	\[
	\dist(\xsol_{y,\lambda},\avX) \leq \frac{1-\sqrt{1-4ab\sigma}}{2a} \leq 2b\sigma .
	\]
\end{proof}

\subsection{Convergence rate for Gaussian measurement maps}\label{sec:errorbnd-gauss}
\subsubsection{Construction of a ``dual'' certificate}
The non-degenerate source condition \eqref{eq:nondeg} is a geometric condition, which is not easy to check in practice and exhibiting a valid vector $q$ is not trivial for general $A$. We will now describe a particular construction of a good candidate (the so-called linearized pre-certificate). Moreover, when $A$ is a Gaussian map, we will also provide sufficient bounds on $m$ needed for conditions \eqref{eq:nondeg}-\eqref{eq:restinj} to hold with overwhelming probability. 

In the sequel, $A$ is a Gaussian map with \iid entries sampled from $\calN(0,1/m)$. From now on, we will focus on the case where $D=\Id_n$ to avoid tedious and unnecessary computations. To lighten the notation, we denote $\T$ and $\e{}$ the model parameters of $R$ at $\avx$.


Following the same notation as in the proof of Theorem~\ref{thm:ratestablependet}, we define the vector
\[
w \eqdef \BAx^\top \argmin_{\BAx^\top q \in\Aff(\partial R(\avx))} \normm{q} .
\]
This amounts to forming $w$ by picking up a specific vector $q$: the one with minimal norm which is unique. If $\BAx$ is injective on $\T$ (which is equivalent to \eqref{eq:restinj} as $D=\Id_n$), it can be shown, using the definition of the model subspace $\T$, that $w$ can be equivalently expressed in closed form as (see \eg \cite{vaiter2015low})
\[
w=\BAx^\top q  \qwhereq q \eqdef {\BAx}_{\T}^{+,\top} \e{} \qandq {\BAx}_{\T}^+ = \Ppa{{\BAx}_{\T}^\top{\BAx}_{\T}}^{-1}{\BAx}_{\T}^\top .
\]
In view of \eqref{eq:ridualcertstrong}, verifying \eqref{eq:nondeg} amounts to ensuring that 
\[
\sigma_{\calC}(w_\S) < 1 .
\]
This is our goal in the rest of the paper where we will provide sample complexity bounds under which one can ensure that this holds with high probability for Gaussian maps.

Our approach is inspired by that of~\cite{candes2011simple}. The key ingredient is the fact that, owing to the isotropy of the Gaussian ensemble, the actions on $\T$ and $\S$ are independent. However, unlike the linear case, in the phase retrieval problem, there is a major issue since $\BAx$, hence the the multiplier $q$, depends on $A\avx$ and there on $\T$ and $\S$. This difficulty can alleviated if the regularizer is such that $\avx \in \T$. This turns out that many important cases including the Lasso, group Lasso penalties, their ordered weighted versions \cite{CandesOWL15}, and many others.

From now on, we assume that $\avx \in \T$. We can then write
\[
\BAx = \diag{|A_\T\avx|}A,
\]
and thus
\[
w=A^\top \eta \qwhereq \eta \eqdef \diag{|A_\T\avx|^2}A_\T\Ppa{A_{\T}^\top\diag{|A_\T\avx|^2}A_{\T}}^{-1}e.
\]
Clearly, isotropy of the Gaussian ensemble entails that $\eta$ and $A_\S$ are independent, which allows us, given the value of $\eta$, to infer the distribution of $A_\S^\top\eta$ with no knowledge of the values of $A_\T$.
Thus, for some $\tau > 0$ and $\kappa \leq 1$, we need to bound 
\begin{equation}\label{eq:probnondeg}
\Pr\Ppa{\sigma_\calC(w_\S) \geq \kappa} = \Pr\Ppa{\sigma_\calC(A_\S^\top\eta) \geq \kappa} \leq \Pr\Ppa{\sigma_\calC(w_\S) \geq \kappa \Big\vert \normm{\eta} \leq \tau} + \Pr\Ppa{\normm{\eta} \geq \tau} .
\end{equation}
The first term in this inequality will be bounded on a case-by-case basis (see the following sections) and uses the fact that conditionally on $\eta$, the entries of $w = A^\top\eta$ are $\iid$ $\calN(0,\normm{\eta}^2/m)$. \\

Let us first consider the second term. We have the following.
\begin{lemma}\label{lem:conc-eta}{\ } 
	Let $\vrho \in ]0,1[$. If $m\geq C(\vrho)\dim(\T)\log(m)$, on the same event we have that $\BAx$ is injective on $\T$, \ie \eqref{eq:restinj} holds, that
	\begin{equation}\label{eq:bnd-norm-eta}
	\normm{\eta}<\frac{1+\delta}{1-\vrho}\normm{e}, 
	\end{equation}
	and that
	\begin{equation}\label{eq:bnd-norm-q}
	\normm{q}<\frac{\normm{\e{}}}{1-\vrho}\frac{\sqrt{m}}{\normm{\avx}},    
	\end{equation}
	with a probability at least $1-\frac{6}{m^2}-e^{-\delta^2/2}$.
\end{lemma}
We refer to Section~\ref{pr:lem:conc-eta} for the proof. The last term of the probability can be made decreasing with $m$ with a stringent upper-bounds on $\eta$ and $q$. This is precisely the aim of the second statement of Lemma~\ref{lem:concent-inj}.

\subsubsection{Bound for a symmetric strong gauge of a polytope} 
We now consider $R$ to be symmetric strong gauge of a polytope, \ie  $R = \gamma_{\calC}$ where $\calC$ is a polytope containing the origin as an interior point and $\e{\avx} \in \ri(\partial R(\avx))$ (see Definition~\ref{def:Rstrong}).
We use the shorthand notation $\calV_\S$ for the set of vertices of $\proj_{S}\calC$.
\begin{lemma}\label{lem:errbndgage}
	Let $\kappa \in]0,1[$, $\vrho > 0$ small enough and $\delta > 0$. Let us define $\beta\eqdef\frac{1+\delta}{1-\vrho}\frac{\normm{e}\max_{v\in\calV_\S}\normm{v}}{\kappa}$. Assume that the number of samples $m$ is such that 
	\[
	m\geq\max\Ppa{C(\vrho)\dim(\T)\log(m),2\beta^2(1+\zeta)\log(\abs{\calV_\S})}
	\] 
	for some $\zeta>0$. Then with probability at least $1-\abs{\calV_\S}^{-\zeta}-\frac{6}{m^2}-e^{-\frac{\delta^2}{2}}$, \eqref{eq:nondeg} and \eqref{eq:restinj} hold, and in particular
	\[
	\sigma_\calC(w_\S)<\kappa .
	\]
\end{lemma}
The proof is in Section~\ref{pr:errbndgage}.\\

Let us now instantiate Lemma~\ref{lem:errbndgage} and Theorem~\ref{thm:ratestablependet} for some popular regularizers.

\paragraph{Example: Lasso}
Denote $I=\supp(\avx)$. For the $\ell_1$-norm, we have $\sigma_\calC = \norm{\cdot}{\infty}$ and $\normm{e}^2=\dim(\T)=s=|I|$. Moreover, $\calV$ is formed by $(\pm a_i)_{i \in [n]}$, where $a_i$ are the vectors of the standard basis. Thus 
\[
\abs{\calV_{\S }}=2(n-s)\leq 2n, \quad \max_{v\in\calV_\S}\normm{v}=1 \qandq \beta=\frac{1+\delta}{1-\vrho}\frac{\sqrt{s}}{\kappa}.
\] 
According to Lemma~\ref{lem:errbndgage}, taking 
\begin{equation}\label{eq:mbndl1convrate}
m\geq\max\Ppa{C(\vrho)s\log(m),2\kappa^{-2}\Ppa{\frac{1+\delta}{1-\vrho}}^2(1+\zeta)s\log(2n)}
\end{equation}
we have
\[
\text{${\BAx}_I$ is injective}  \qandq \norm{w_\S}{\infty} < \kappa 
\]
with probability at least $1-(2n)^{-\zeta}-\frac{6}{m^2}-e^{-\frac{\delta^2}{2}}$. Thus $m \gtrsim s\log(n)$ measurements are sufficient for \eqref{eq:nondeg} and \eqref{eq:restinj} to hold when $\avx$ is $s$-sparse. In turn, the error bound Theorem~\ref{thm:ratestablependet} holds. We need to estimate the constant in that bound. According to the proof of Theorem~\ref{thm:ratestablependet}, this constant is $2b$ where
\[
b = \normm{{\BAx}_{T}^+}\Ppa{2+c\normm{q} } + \alpha \frac{\Ppa{2+\frac{c}{2}\normm{q}}^2}{2c}
\qwithq
\alpha = \frac{\anormOP{\proj_{\S}}{1}{2}+\normm{{\BAx}_{T}^+}\norm{A\avx}{\infty}\anormOP{A_{\S}}{1}{2}}{1-\sigma_\calC\Ppa{w_{\S}}} .
\]
We have just shown that
\[
\frac{1}{1-\norm{w_{\S}}{\infty}} \leq \frac{1}{1-\kappa}
\]
with high probability when $m$ verifies the bound \eqref{eq:mbndl1convrate}. On this same event, we also have (see Lemma~\ref{lem:concent-inj}, Lemma~\ref{lem:conc-eta} and its proof),
\begin{equation}\label{eq:ABql1}
\norm{A_\T\avx}{\infty} \leq \frac{1+\delta}{\sqrt{m}}\normm{\avx}, \quad \normm{{\BAx}_{{\T}}^+} \leq \frac{\sqrt{m}}{(1-\vrho)\normm{\avx}} \qandq \normm{q}<\frac{\sqrt{s m}}{(1-\vrho)\normm{\avx}} .
\end{equation}
To complete our analysis, we need to bound $\anormOP{\proj_{\S}}{1}{2}$ and $\anormOP{A_{\S}}{1}{2}$. 

Obviously
\[
\anormOP{\proj_{\S}}{1}{2} = \sup_{x} \frac{\normm{x_{I^c}}}{\norm{x_{I^c}}{1}} = 1 .
\]
In addition, we have for any $x \in \bbR^n$
\[
\frac{\normm{\sum_{i \in I^c}A_i x[i]}}{\norm{x_{I^c}}{1}} \leq \max_{i \in I^c} \normm{A_i} \frac{\sum_{i \in I^c} |x[i]|}{\norm{x_{I^c}}{1}} = \max_{i \in I^c} \normm{A_i} ,
\]
whence we get the upper bound
\[
\anormOP{A_{\S}}{1}{2} \leq \max_{i \in I^c} \normm{A_i} .
\]
By a union bound and Proposition~\ref{pro:gauss-space}, it is immediate to show that
\[
\max_{i \in I^c} \normm{A_i} \leq 1 + \sqrt{\frac{2t\log(n)}{m}}, \text{ for some } t > 1 ,
\]
with probability at least $1-n^{1-t}$. 

We now proceed to state our estimation error bound for the Lasso. Replacing all the terms in $b$ by their estimates, and plugging into Theorem~\ref{thm:ratestablependet}, we get the following.
\begin{proposition}\label{pro:convrate-l1} 
Let $\kappa \in]0,1[$, $\vrho > 0$ small enough, $\delta > 0$ and $t > 1$. Consider the noisy phaseless measurements in \eqref{eq:GeneralPR} where $\avx$ is $s$-sparse and $\sigma \eqdef \normm{\epsi}$ is small enough. Consider problem \eqref{eq:eqminregulPRpen} where $R$ is the $\ell_1$-norm and $\lambda = c\sigma$ for $c > 0$.
If $A$ is a Gaussian map with \iid entries from $\calN(0,1/m)$ with $m$ verifying \eqref{eq:mbndl1convrate}, then with probability at least $1-n^{1-t}-(2n)^{-\zeta}-\frac{6}{m^2}-e^{-\frac{\delta^2}{2}}$, any minimizer $\xsol_{y,\lambda}$ of \eqref{eq:eqminregulPRpen} satisfies
\begin{multline}\label{eq:bnd-dist-l1}
\dist(\xsol_{y,\lambda},\avX) \leq \Bigg(\frac{2\sqrt{m}}{(1-\vrho)\normm{\avx}}\Ppa{2+c\frac{\sqrt{sm}}{(1-\vrho)\normm{\avx}}} \\
+ \frac{1+\frac{1+\delta}{1-\vrho}\Ppa{1+\sqrt{\frac{2t\log(n)}{m}}}}{1-\kappa}\frac{\Ppa{2+\frac{c}{2}\frac{\sqrt{sm}}{(1-\vrho)\normm{\avx}}}^2}{2c}\Bigg) \sigma .
\end{multline}
\end{proposition}

The performance guarantee \eqref{eq:bnd-dist-l1} has an alternative form in terms of the signal-to-noise ratio (SNR). The latter is captured for the model \eqref{eq:GeneralPR} by
\[
\SNR \eqdef \frac{\norm{A\avx}{4}^4}{\normm{\epsi}^2} \approx \frac{3\normm{\avx}^4}{m\sigma^2} ,
\]
where we used that $A$ has \iid $\calN(0,1/m)$ entries. We then have the bound
\begin{equation*}
\dist(\xsol_{y,\lambda},\avX) \lesssim \frac{\normm{\avx}}{\sqrt{\SNR}}\Ppa{\Ppa{1+\frac{\sqrt{s}}{m^{-1/2}\normm{\avx}}}
+ \Ppa{1+\Ppa{1+\frac{\sqrt{s}}{m^{-1/2}\normm{\avx}}}^2}m^{-1/2}\normm{\avx}}
\end{equation*}
revealing the stability of minimizers of \eqref{eq:eqminregulPRpen} as a function of SNR.

\subsubsection{Bound for the group Lasso}
Recall the group Lasso penalty (Section~\ref{sec:recoverynoiselessgaussiandecomp}) with $L$ blocks of equal size $B$. Let $I=\bsupp(\avx)$. For the $\ell_1-\ell_2$-norm, we have $\sigma_\calC(x) = \max_{i \in I}\normm{x[b_i]}$ and, $\normm{e}^2=s$ where $s=|I|$ is the number of active blocks in $\avx$. Let $\Lambda=\bigcup_{i \in I} b_i$, and $\Lambda^c$ its complement. We then have $\dim(\T)=|\Lambda|=sB$.

\begin{lemma}\label{lem:errbndl1l2}
Let $\kappa \in]0,1[$, $\vrho > 0$ small enough and $\delta > 0$. Assume that the number of samples $m$ is such that 
\begin{equation}\label{eq:mbndl1l2convrate}
m\geq\max\Ppa{C(\vrho)sB\log(m),2\kappa^{-2}(1+\zeta)\frac{1+\delta}{1-\vrho}s\Ppa{\sqrt{\log(L)}+\sqrt{B}}^2}
\end{equation}
form some $\zeta>0$. Then with probability at least $1-L^{-\zeta}-\frac{6}{m^2}-e^{-\frac{\delta^2}{2}}$, \eqref{eq:nondeg} and \eqref{eq:restinj} hold, and in particular
\[
\max_{i \in I}\normm{w[b_i]} < \kappa .
\]
\end{lemma}
The proof is in Section~\ref{pr:errbndl1l2}.\\

Thus $m \gtrsim s\Ppa{\Ppa{\sqrt{\log(L)}+\sqrt{B}}^2+B\log(m)}$ measurements are sufficient for \eqref{eq:nondeg} and \eqref{eq:restinj} to hold for the group Lasso when $\avx$ is $s$-group sparse. Therefore, the error bound in Theorem~\ref{thm:ratestablependet} holds. We now estimate the corresponding constant $2b$ as above.
By the triangle inequality, we can upper-bound
\[
\frac{\normm{x[\Lambda^c]}}{\norm{x[\Lambda^c]}{1,2}} \leq  \frac{\sum_{i \in I^c}\normm{x[b_i]}}{\sum_{i \in I^c}\normm{x[b_i]}} = 1,
\]
and thus
\[
\anormOP{\proj_{\S}}{(1,2)}{2} = \sup_{x} \frac{\normm{x[\Lambda^c]}}{\norm{x[\Lambda^c]}{1,2}} = 1.
\]
Moreover, for any $x \in \bbR^n$
\[
\frac{\normm{\sum_{i \in I^c}A_{b_i}x[b_i]}}{\sum_{i \in I^c}\normm{x[b_i]}} \leq \frac{\sum_{i \in I^c} \normm{A_{b_i}}\normm{x[b_i]}}{\sum_{i \in I^c}\normm{x[b_i]}} \leq \max_{i \in I^c} \normm{A_{b_i}} .
\]
This yields
\[
\anormOP{A_{\S}}{1,2}{2} = \anormOP{A_{\Lambda^c}}{(1,2)}{2}  \leq  \max_{i \in I^c} \normm{A_{b_i}} .
\]
Invoking standard concentration inequalities of the largest singular value of $A_{b_i}$ and a union bound, we get
\[
\max_{i \in I^c} \normm{A_{b_i}} \leq 1 + \sqrt{B/m} + \sqrt{2t\log(L)/m}, \text{ for some } t > 1 ,
\]
with probability at least $1-L^{1-t}$. The bounds in \eqref{eq:ABql1} are also valid for the group Lasso under our sample complexity bound. Combining this discussion with Theorem~\ref{thm:ratestablependet} and Lemma~\ref{lem:errbndl1l2}, we have proved the following.
\begin{proposition}\label{pro:convrate-l1l2} 
Let $\kappa \in]0,1[$, $\vrho > 0$ small enough, $\delta > 0$ and $t > 1$. Consider the noisy phaseless measurements in \eqref{eq:GeneralPR} where $\avx$ is $s$-group sparse and $\sigma \eqdef \normm{\epsi}$ is small enough. Consider problem \eqref{eq:eqminregulPRpen} where $R$ is the $\ell_1-\ell_2$ norm and $\lambda = c\sigma$ for $c > 0$.
If $A$ is a Gaussian map with \iid entries from $\calN(0,1/m)$ with $m$ verifying \eqref{eq:mbndl1l2convrate}, then with probability at least $1-L^{1-t}-L^{-\zeta}-\frac{6}{m^2}-e^{-\frac{\delta^2}{2}}$, any minimizer $\xsol_{y,\lambda}$ of \eqref{eq:eqminregulPRpen} satisfies
\begin{multline}\label{eq:bnd-dist-l1l2}
\dist(\xsol_{y,\lambda},\avX) \leq \Bigg(\frac{2\sqrt{m}}{(1-\vrho)\normm{\avx}}\Ppa{2+c\frac{\sqrt{sm}}{(1-\vrho)\normm{\avx}}} \\
+ \frac{1+\frac{1+\delta}{1-\vrho}\Ppa{1 + \sqrt{B/m} + \sqrt{2t\log(L)/m}}}{1-\kappa}\frac{\Ppa{2+\frac{c}{2}\frac{\sqrt{sm}}{(1-\vrho)\normm{\avx}}}^2}{2c}\Bigg) \sigma .
\end{multline}
\end{proposition}
One can get also get an alternative form of the bound terms of the SNR exactly as we did for the Lasso. We leave the details to the reader.

\section{Numerical Experiments}\label{sec:numexp}
In this section, we discuss some numerical experiments to give give a gist of our results. For this, we will use a Bregman proximal gradient (BPG) algorithm, to be described shortly, to solve \eqref{eq:eqminregulPRpen}. However, we would like to stress the fact that our results are on the minimizers of \eqref{eq:eqminregulPRpen} and not on the outcome of BPG. In particular, BPG is guaranteed to converge to a critical point of \eqref{eq:eqminregulPRpen} in general. But it is still an open problem whether $m$ as large as devised by our bounds, typically scaling linearly in the intrinsic dimension of $\avx$, is sufficient for BPG to provably converge to a global minimizer (which is known by our results to be close to $\pm \avx$ for small noise). In fact, our numerical experiments suggest that more measurements are actually needed for BPG without a particular initialization to stably recover $\pm\avx$ from small noise.


\subsection{Bregman Proximal Gradient}
Let us recall that in \eqref{eq:eqminregulPRpen}, the smooth part (that we denote $f$) in the objective $F$ is semi-algebraic $C^2(\bbR^n)$, but is non-convex (though it is weakly convex). Besides, its gradient is not Lipschitz continuous. Therefore, we associate to the smooth part $f$ the kernel function in \eqref{eq:entropy}.
$\psi\in C^2(\bbR^n)$ has full domain and $1-$strongly convex function with a gradient that is Lipschitz over bounded subsets of $\bbR^n$. Moreover, $f$ is smooth relative to $\psi$ \cite{bolte_first_2017}, \ie 
\begin{equation}\label{eq:relsmooth}
\exists L > 0 \text{ such that }  D_{f}(x,z) \leq L D_\psi(x,z), \forall x, z \in \bbR^n .
\end{equation}
The constant $L$ can be computed explicitly and sharp estimates were proposed in \cite{godeme2023provable} in the oversampling regime. Problem \eqref{eq:eqminregulPRpen} is  amenable to the Bregman Proximal Gradient (BPG) as proposed in \cite{bolte_first_2017} and further studied for the phase retrieval without regularization in \cite{godeme2023provable,godeme2024stable}. Its main steps are summarized in Algorithm~\ref{alg:BPGBT}.

 \begin{algorithm}[H]
 	\caption{Bregman Proximal Gradient}
 	\label{alg:BPGBT}
 	\textbf{Initialization:} $\xo\in\bbR^n$ \;
 	\For{$k=0,1,\ldots$}{
 		\vspace{-0.25cm}
 		\begin{flalign} \tag{BPG}\label{BPGalgo} 
 			&\begin{aligned} 
 				&\xkp = \Ppa{\nabla\psi+\gamma\lambda\partial R}^{-1}\Ppa{\nabla\psi(\xk)-\gamma\nabla F(\xk)}, \quad \gamma<\frac{1}{L};  
 			\end{aligned}&
 		\end{flalign}
 	}
 \end{algorithm}
 
Since the regularizer $R$ in \eqref{eq:eqminregulPRpen} is proper, lsc and convex, the operator $\Ppa{\nabla\psi+\gamma\lambda\partial R}^{-1}$, which is nothing but the Bregman proximal operator, is non-empty and single-valued over the whole space.

\subsection{Results}
Throughout our experiments, $A$ is drawn from the Gaussian ensemble with \iid $\calN(0,1/m)$ entries. We take $n=128$ and we run Algorithm~\ref{alg:BPGBT} with a constant step-size $\gamma=\frac{0.99}{3+10^{-4}}$. For $R$, we have tested several regularizers as described hereafter.

\begin{example}[Lasso] 
For the $\ell_1$ norm, the Bregman proximal mapping with our entropy $\psi$ has a nice formula \cite[Proposition~5.1]{bolte_first_2017}. The underlying vector $\avx$ is taken to be sparse with $s=12$ non-zeros entries. The number of quadratic measurements is taken as $m=0.5\times s^{1.5}\times\log(n)$, which grows larger than linear with $s$. As there is no noise in this experiment, we took $\lambda=10^{-8}$. Figure~\ref{fig: solving_l1} shows the recovery results. The left plot of Figure~\ref{fig: solving_l1} displays the relative error of the iterates vs the number of iterations. On the right plot, we display the cardinality of the support of the iterates. Clearly, the left plot shows that Algorithm~\ref{alg:BPGBT} identifies the correct support after 300 iterations and converges to $\pm\avx$.
\begin{figure}[h]
\centering
\includegraphics[trim={0cm 8cm 0 8cm},clip,width=0.7\linewidth]{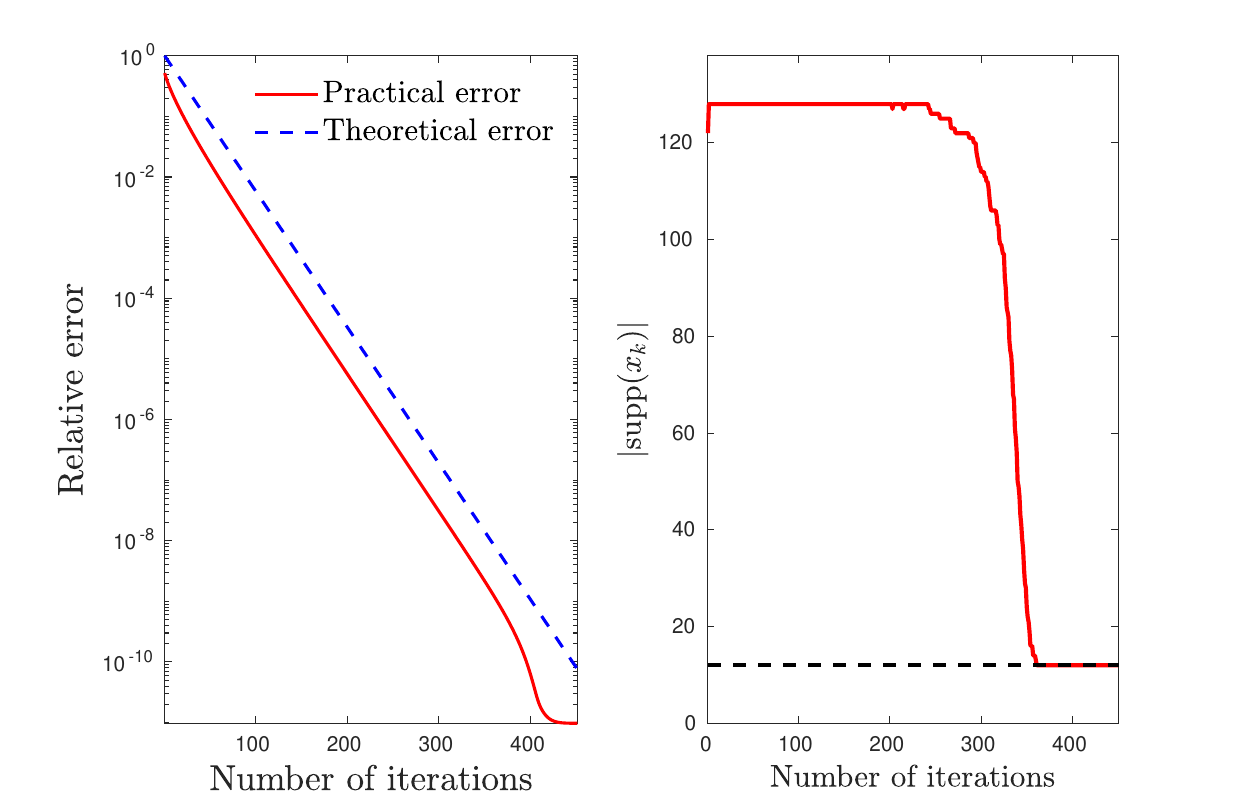}
 \caption{Phase retrieval with the Lasso ($\ell_1-$norm) regularizer.}
\label{fig: solving_l1}
\end{figure}

The left plot of Figure~\ref{fig:stability} depicts the evolution of the estimation error for the Lasso phase retrieval as a function of the noise level $\sigma$ (small). For each value of $\sigma$, we choose $\lambda=3\sigma$. As expected, the error scales linearly with the noise.

\begin{figure}[h]
\centering
\includegraphics[width=0.45\linewidth]{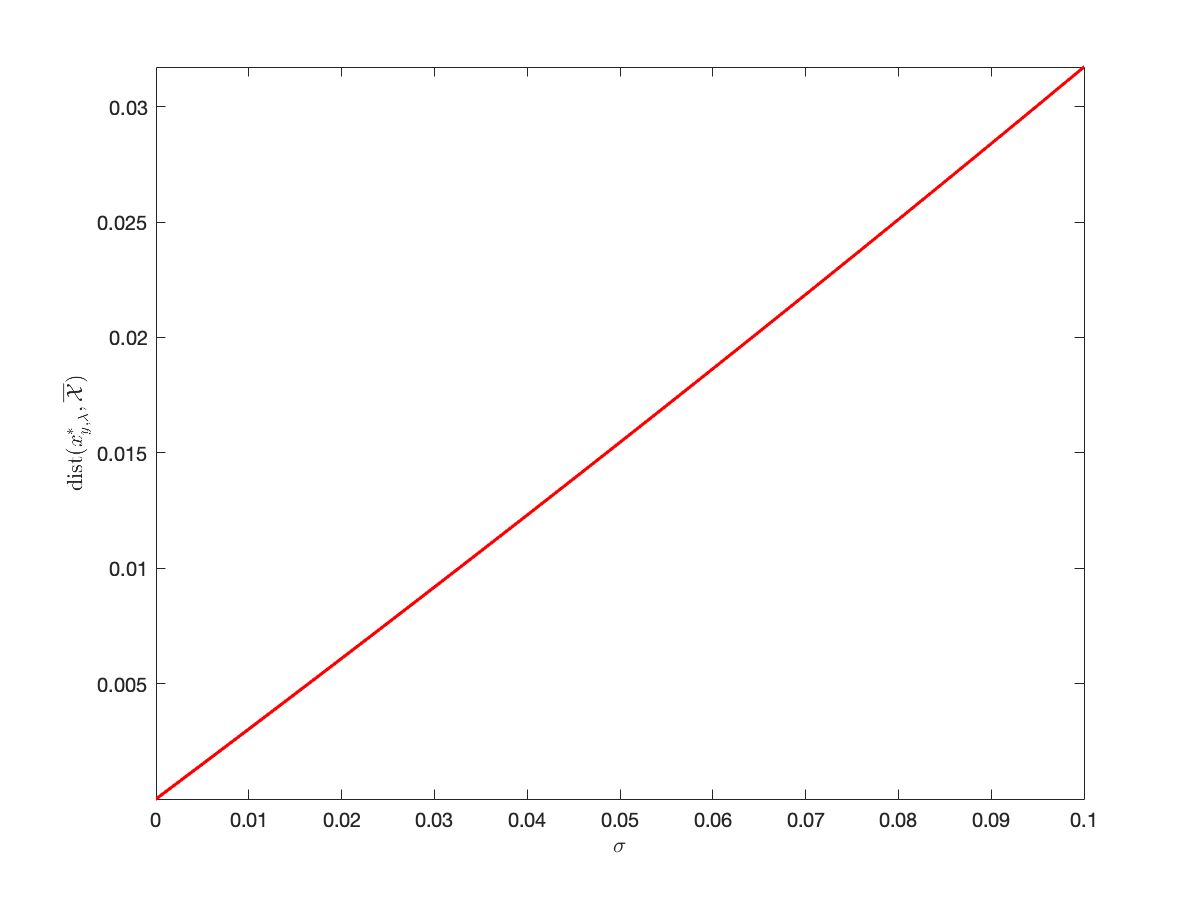}
\includegraphics[width=0.45\linewidth]{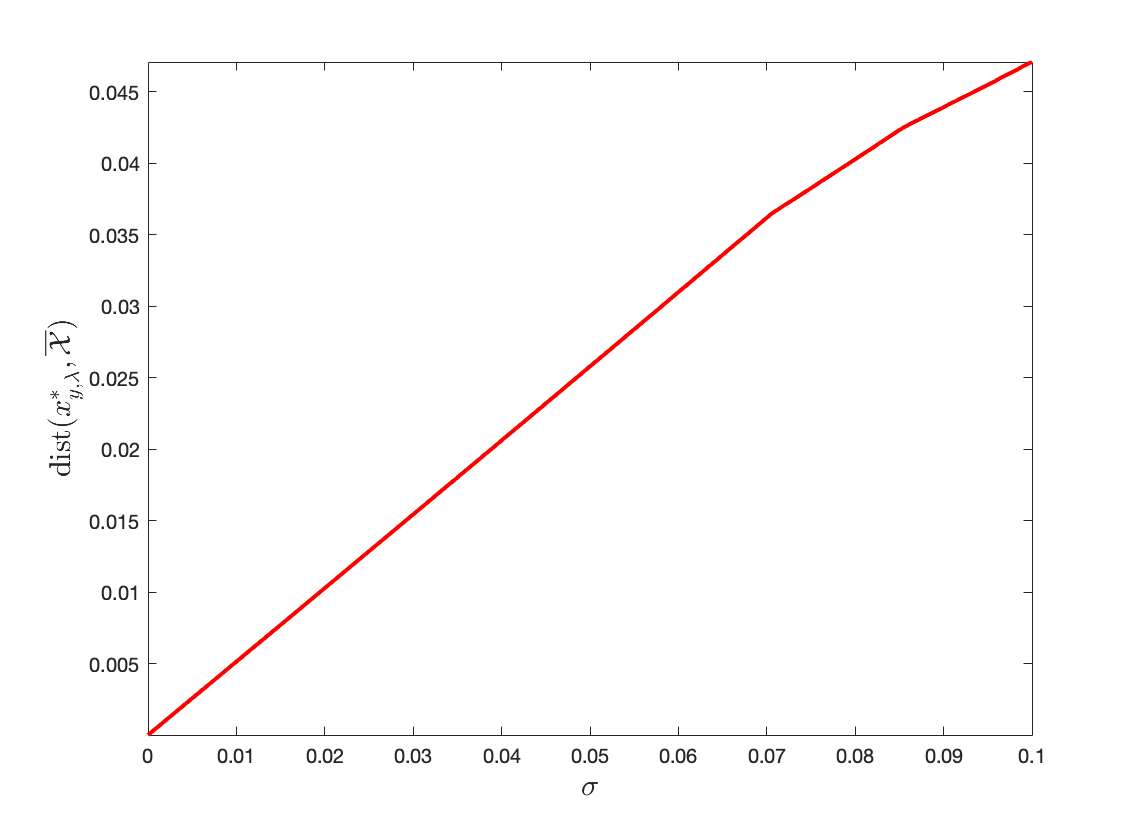}
\caption{Stable phase retrieval with the Lasso (left) and TV (right).}
\label{fig:stability}
\end{figure}

\end{example}

\begin{example}[Group Lasso] Here, we take $R$ as the group/block Lasso which is designed to promote group sparsity. By extending the result of \cite[Proposition~5.1]{bolte_first_2017}, using the specific structure of the $\ell_1-\ell_2$ norm, the Bregman proximal mapping with our entropy $\psi$ in this case turns out also to have a closed formula.
In our experiment, we consider $\avx$ to have $2$ non-zero blocks of size $B=8$ each. The number of quadratic measurements is $m=0.5\times(2 \times 8)^2\times\log(128)$. We also take $\lambda=10^{-8}$ as no noise was added in this experiment. The results are shown in Figure~\ref{fig: solving_l12}, and they are consistent with the discussion for the \onenorm.
\begin{figure}[h]
\centering
\includegraphics[trim={0cm 8cm 11cm 8.75cm},clip,width=0.35\linewidth]{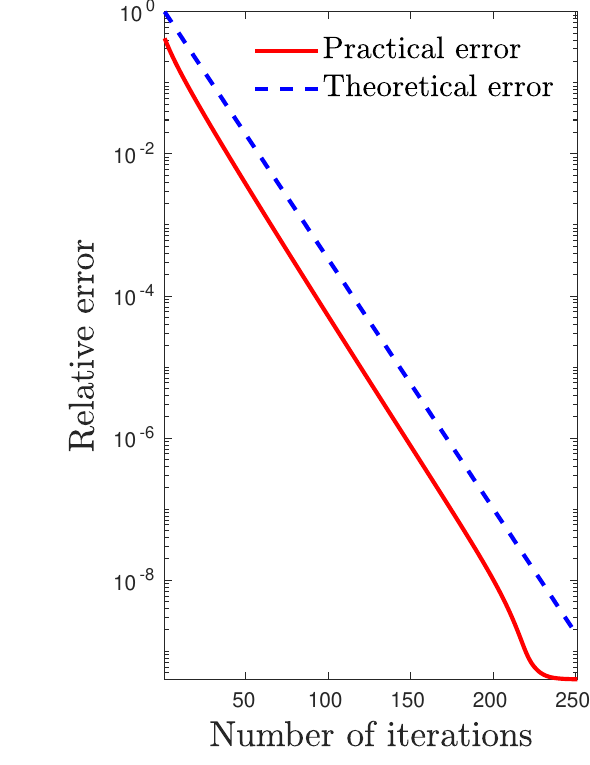}
 \caption{Phase retrieval with the group Lasso ($\ell_1-\ell_2$ norm) regularizer.}
\label{fig: solving_l12}
\end{figure}
\end{example}

\begin{example}[TV regularizer] 

In this experiment, the original vector $\avx$ is piecewise constant with $s=12$ randomly placed jumps. The number of measurements is $m=0.5\times s^2\times\log(n)$. The regularizer is the total variation (TV). However the Bregman proximal mapping of TV does not have an explicit expression. Therefore, we used an inner iteration to compute it using an accelerated proximal gradient algorithm on the dual \cite{FadiliDualTV10}. The results are depicted in Figure~\ref{fig: solving_tv}. The left plot shows the original (dashed line) and the recovered vector (solid line). The right plot shows the evolution of the relative error vs iterations. 
\begin{figure}[h]
\centering
\includegraphics[trim={0cm 7cm 0cm 7cm},clip,width=.75\linewidth]{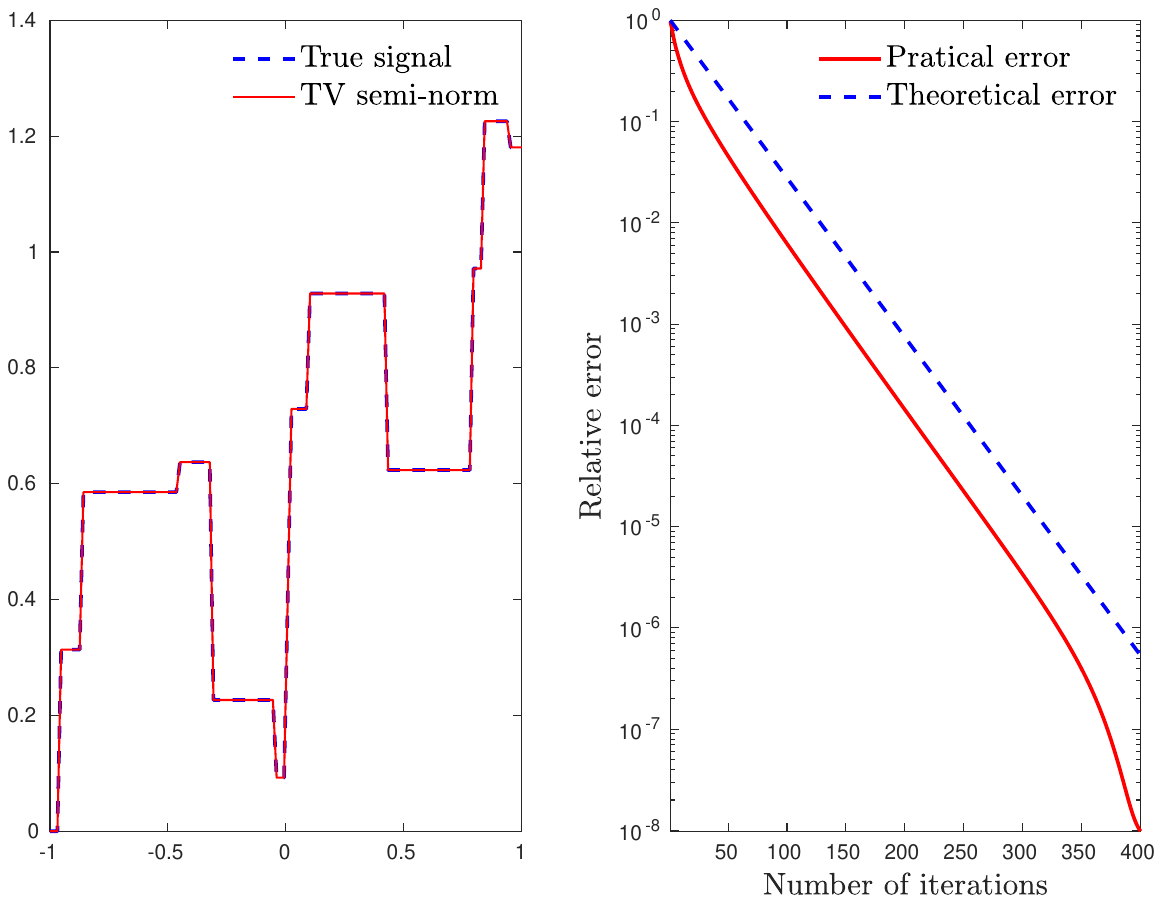}
 \caption{Phase retrieval with the TV semi-norm.}
\label{fig: solving_tv}
\end{figure}
\end{example}

The right plot of Figure~\ref{fig:stability} depicts the evolution of the estimation error for the TV phase retrieval as a function of the noise level $\sigma$. We choose $\lambda=3\sigma$. One again sees that the error scales linearly with the noise.

\begin{example}\textbf{(Wavelet synthesis-type prior).} 
We here cast the phase retrieval problem as
\begin{equation}\label{eq:tv_synthesis}
\min_{v \in \bbR^p} \frac{1}{4m}\normm{y-|AW v|^2}^2+\lambda \norm{v}{1}, \quad \lambda>0 ,
\end{equation}
where $W$ is a wavelet synthesis operator. The reconstructed vector is given by $x=W v$. When $W$ is orthonormal, this is equivalent to the analysis-type formulation with $D=W^\top$. This is not anymore the case when $W$ is redundant. 

In this experiment, we will use the shift-invariant wavelet dictionary with the Haar wavelet, which is closely related to the TV regularizer for 1D signals; see \cite{steidl2004equivalence}. We take the same number of jumps and measurements as in the previous example. The results are shown in Figure~\ref{fig: solving_tv_syn}.

\begin{figure}[h]
\centering
\includegraphics[trim={0cm 7cm 0cm 7cm},clip,width=.75\linewidth]{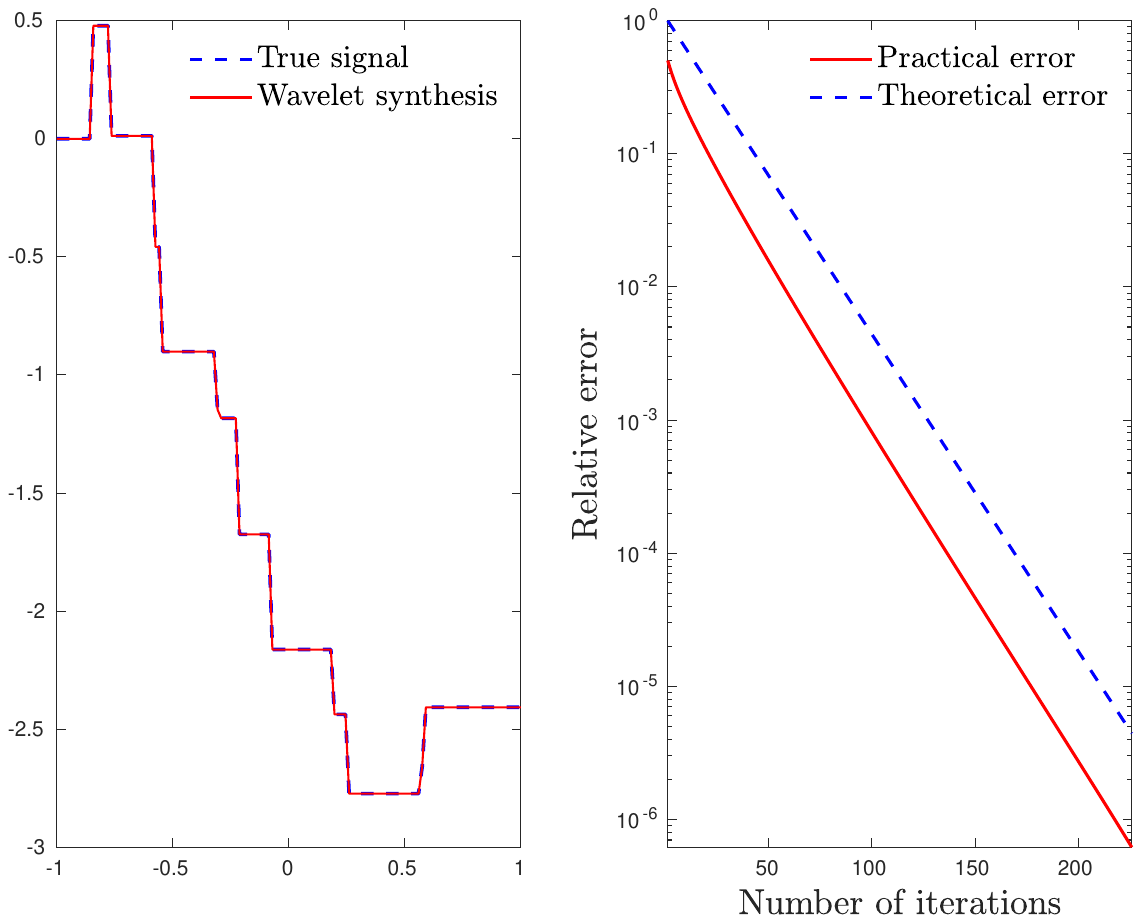}
\caption{Phase retrieval with the wavelet-synthesis prior formulation.}
\label{fig: solving_tv_syn}
\end{figure}
\end{example}


\begin{appendices}
	
\section{Proofs for Section~\ref{sec:errorbnd-gauss}}
\subsection{Proof of Lemma~\ref{lem:conc-eta}}\label{pr:lem:conc-eta}
We have
\begin{align*}
	\normm{\eta}
	&=\normm{\diag{|A_\T\avx|}{\BAx}_{\T}^{+,\top} \e{}} \\
	&\leq \norm{A_\T\avx}{\infty} \normm{{\BAx}_{\T}^{+}} \normm{\e{}} \\
	&=\norm{A_\T\avx}{\infty} \lambda_{\min}\Ppa{{\BAx}_{\T}^\top{\BAx}_{\T}}^{-1/2} \normm{\e{}} \\
	&=\norm{A_\T\avx}{\infty} \lambda_{\min}\Ppa{A_{\T}^\top\diag{|A_\T\avx|^2}A_{\T}}^{-1/2} \normm{\e{}} .
\end{align*}
By Lemma~\ref{lem:concent-inj}, we have
\[
\norm{A_\T\avx}{\infty} \leq \frac{1+\delta}{\sqrt{m}}\normm{\avx}
\]
with probability at least $1-e^{-\frac{\delta^2}{2}}$. Observe also that
\[
A_{\T}^\top\diag{|A_\T\avx|^2}A_{\T} = \sum_{r=1}^m{|\pscal{(a_r)_\T,\avx}|^2(a_r)_\T\transp{(a_r)_\T}} .
\]
It then follows from Lemma~\ref{lem:concent-hess} that
\[
\lambda_{\min}\Ppa{A_{\T}^\top\diag{|A_\T\avx|^2}A_{\T}} \geq \frac{1-\vrho}{m}\normm{\avx}^2 \geq \frac{(1-\vrho)^2}{m}\normm{\avx}^2
\]
with probability at least $1-\frac{6}{m^2}$ as soon as $m \geq C(\vrho)\dim(\T)\log(m)$. Thus,
\[
\Pr\Ppa{\normm{\eta} \geq \frac{1+\delta}{1-\vrho}\normm{\e{}}} \leq\frac{6}{m^2}+e^{-\frac{\delta^2}{2}}.
\] 

Since $q={\BAx}_{\T}^{+,\top} \e{}$, we have
\begin{align*}
	\normm{q}&\leq \normm{{\BAx}_{\T}^{+}} \normm{\e{}} .
\end{align*}
We then argue similarly to above by invoking again Lemma~\ref{lem:concent-hess}.


\subsection{Proof of Lemma~\ref{lem:errbndgage}}\label{pr:errbndgage}

Let us first observe that $\sigma_{\calC}(A_\S^\top\eta)=\max_{v\in\calC}\pscal{A_\S^\top\eta,v}$. Therefore 
\begin{align*}
	\Pr\Ppa{\sigma_{\calC}(A_\S^\top\eta)\geq \kappa \Big\vert \normm{\eta} \leq \tau} 
	&=\Pr\Ppa{\max_{v\in\calC}\pscal{A_\S^\top\eta,v}\geq \kappa \Big\vert \normm{\eta} \leq \tau},\\
	&=\Pr\Ppa{\max_{v\in\calV_\S}\pscal{A^\top\eta,v}\geq \kappa \Big\vert \normm{\eta} \leq \tau},\\
	&\leq \abs{\calV_\S}\max_{v\in\calV_\S}\Pr\Ppa{\pscal{Z\eta,v} \geq \kappa\sqrt{m} \Big\vert \normm{\eta} \leq \tau},
\end{align*}
where $Z \in \bbR^{n \times m}$ is drawn from the standard Gaussian ensemble. Let us observe that $Z\mapsto\pscal{Z\eta,v}$ is a Lipschitz continuous function of constant $\normm{\eta}\normm{v} \leq \tau $. Using Proposition~\ref{pro:gauss-space}, we get
\begin{align*}
	\Pr\Ppa{\sigma_{\calC}(A_\S^\top\eta)\geq \kappa \Big\vert \normm{\eta} \leq \tau}&\leq \abs{\calV_\S}\max_{v\in\calV_\S}e^{-\frac{m\kappa^2}{2\tau^2\normm{v}^2}},\\
	&\leq\abs{\calV_\S}e^{-\frac{m\kappa^2}{2\tau^2\max_{v\in\calV_\S}\normm{v}^2}},\\
	&=e^{-\frac{m\kappa^2}{2\tau^2\max_{v\in\calV_\S}\normm{v}^2}+\log(\abs{\calV_\S})}.
\end{align*}
In view of Lemma~\ref{lem:conc-eta}, we set $\tau=\frac{1+\delta}{1-\vrho}\normm{e}$, and we get the claim using \eqref{eq:probnondeg} with the devised value of $\beta$ and the bound on $m$.
\subsection{Proof of Lemma~\ref{lem:errbndl1l2}}\label{pr:errbndl1l2}
We use a union bound to get 
\begin{align*}
	\Pr\Ppa{\sigma_{\calC}\pa{A_\S^\top\eta} \geq \kappa \Big\vert \normm{\eta} \leq \tau} 
	&=\Pr\Ppa{\max_{i\in I^c}\norm{A_{b_i}^\top\eta}{2}  \geq \kappa \Big\vert \normm{\eta} \leq \tau}\\ 
	&\leq (L-s)\max_{i \in I^c}\Pr\Ppa{\normm{A_{b_i}^\top\eta} \geq \kappa \Big\vert \normm{\eta} \leq \tau} .
\end{align*}
We now argue as in \cite{candes2011simple} observing that conditioned on $\eta$, $\frac{m}{\normm{\eta}^2}\normm{A_{b_i}^\top\eta}^2=\frac{m}{\normm{\eta}^2}\normm{w[b_i]}^2$ are identically distributed as a $\chi$-squared random variable with $B$ degrees of freedom. By concentration of the latter, we have for any $t > 0$
\begin{align*}
	\Pr\Ppa{\frac{\sqrt{m}}{\normm{\eta}}\normm{A_{b_i}^\top\eta} \geq \frac{\sqrt{m}}{\normm{\eta}}\esp{\normm{A_{b_i}^\top\eta}} + t \Big\vert \normm{\eta} \leq \tau}
	= \Pr\Ppa{\frac{\sqrt{m}}{\normm{\eta}}\normm{A_{b_i}^\top\eta} \geq \sqrt{B} + t \Big\vert \normm{\eta} \leq \tau} \leq e^{-t^2/2} .
\end{align*}
Therefore
\begin{align*}
	\Pr\Ppa{\sigma_{\calC}\pa{A_\S^\top\eta} \geq \kappa \Big\vert \normm{\eta} \leq \tau} 
	&\leq L \exp\Ppa{-\Ppa{\sqrt{m}\kappa/\tau - \sqrt{B}}^2/2} .
\end{align*}
In view of Lemma~\ref{lem:conc-eta}, we set $\tau=\frac{1+\delta}{1-\vrho}\normm{e}=\frac{1+\delta}{1-\vrho}\sqrt{s}$. Plugging this  and the last inequality in \eqref{eq:probnondeg}, and the devised bound on $m$, we get the claim.

\section{Concentration inequalities}
We start recalling two standard concentration inequalities. For a random variable $X$ and $k \geq 1$, we define
\[
\norm{X}{\psi_k} = \sup_{p \geq 1} p^{-1/k}(\esp{|X|^p})^{1/p} .
\]
$\norm{X}{\psi_2}$ is known as the sub-Gaussian norm while $\norm{X}{\psi_1}$ is the sub-exponential norm.

\begin{proposition}[Hoeffding-type inequality]\label{pro:vec-hoeffding} Let $X=(X_1,\cdots,X_N)$ be independent centred sub-Gaussian random variables, and let $K=\max_{i}\norm{X_i}{\psi_2}$. Then for every  vector $a\in\bbR^N$ and $t\geq0,$ we have
	\[
	\prob\Ppa{\abs{\pscal{a,X}}\geq t}\leq e.\exp\Ppa{-\frac{ct^2}{K^2\normm{a}^2}}, 
	\]
	where $c>0$ is an absolute constant. 
\end{proposition}

\begin{proposition}[Bernstein-type inequality]\label{pro:vec-bernstein} Let $X_1,\cdots,X_N$ be independent centered sub-exponential random variables, and let $K=\max_{i}\norm{X_i}{\psi_1}$. Then for every  vector $a\in\bbR^N$ and $t\geq0,$ we have
	\[
	\prob\Ppa{\abs{\pscal{a,X}}\geq t}\leq e.\exp\Bba{-c\min\Ppa{\frac{t^2}{K^2\normm{a}^2},\frac{t}{K\norm{a}{\infty}}}}, 
	\]
	where $c>0$ is an absolute constant. 
\end{proposition}

\begin{lemma}\label{lem:concent-inj} Fix $\delta > 0$ and $x\in\bbR^n$. We have, 
	\begin{enumerate}[label=(\roman*)]
		\item
		\begin{equation}
		\norm{A x}{\infty} \leq \frac{1+\delta}{\sqrt{m}}\normm{x}.
		\end{equation}
		This happens with probability at least $1-e^{\frac{-\delta^2}{2}}$. 
		\item Moreover
		\begin{equation}
		\norm{A x}{\infty} \leq \sqrt{(1+\delta)\frac{2\log(m)}{m}}\normm{x}.
		\end{equation}
		with probability at least $1-m^{-\delta}$. 
	\end{enumerate}
\end{lemma}
\begin{proof} 
To show (i), observe that $\norm{A x}{\infty} \leq \normm{A x}$ and then use Proposition~\ref{pro:gauss-space} since $A \mapsto \normm{A x}$ is $\normm{x}$-Lipschitz continuous and $\esp{\normm{Ax}} \leq \normm{x}/\sqrt{m}$. For the second claim, we use a more direct and standard argument. We have by the union bound and the tail bound for a standard Gaussian random variable that
\begin{align*}
\prob\Ppa{\norm{A x}{\infty} \geq \sqrt{(1+\delta)\frac{2\log(m)}{m}}\normm{x}} 
&\leq m\prob\Ppa{|Z| \geq \sqrt{2(1+\delta)\log(m)}}, \qquad Z \sim \calN(0,1) \\
&\leq m e^{-(1+\delta)\log(m)} = m^{-\delta} .  
\end{align*}
\end{proof}

\medskip

Let us consider the model linear subspace $\T\subset\bbR^n$, and denote $d=\dim(T)$.  Throughout this section, we will see $\T$ as $\bbR^d$ since there exists an isometry from $\bbR^d$ onto $T$. In turn, $A_T$ can be viewed as a $m\times d$ matrix whose entries are $\iid$ $\calN(0,1/m)$. We have the following concentrations.
\begin{lemma}\label{lem:concent-hess} 
	Fix $\vrho\in]0,1[$ (small enough) and choose $0<\bar{\vrho}<\frac{\vrho+3}{10\log(m)}$. 
	\begin{enumerate}
		\item[(i)] If the number of samples obeys $m \geq C(\vrho)d\log(d)$, for some sufficiently large $C(\vrho) > 0$, we have
		\begin{align}\label{eq:conc-hess}
			\normm{mA_T^\top\diag{|A_T\avx|^2}A_T-\Ppa{2\avx\transp{\avx}+\normm{\avx}^2\Id}}\leq \vrho \normm{\avx}^2.
		\end{align}
		with a probability at least $1-5e^{-\zeta d}-\frac{4}{d^2}$ where $\zeta$ is a fixed numerical constant.
		\item[(ii)] If the number of samples obeys $m \geq C(\bar{\vrho},\vrho)d\log(m)$, for some sufficiently large $C(\bar{\vrho},\vrho) > 0$, \eqref{eq:conc-hess} holds true with a probability at least $1-\frac{6}{m^2}$.
	\end{enumerate}
\end{lemma}
\begin{proof} 
	The proof of claim $(i)$ is just an application of \cite[Lemma~B.2]{godeme2023provable} to $A_T$. 
	
	For the proof of claim $(ii)$, we have to modify the choice of $m$ in the different concentrations used in the proof of \cite[Lemma~B.2]{godeme2023provable}. We provide here a self-contained proof. We have to emphasize that showing \eqref{eq:conc-hess} is  similar to showing  that 
	\begin{align}\label{eq:bndsspectmtxgaussian}
		\normm{\som{|\bar{a}_r[1]|^2\bar{a}_r \transp{\bar{a}_r}-\Ppa{2e_1\transp{e_1}+\Id}}}\leq {\vrho},
	\end{align}
	where the entries of $\bar{a}_r$ are now standard Gaussian random variable and $e_1$ is a vector of the standard basis.
	
	From symmetry arguments, showing \eqref{eq:bndsspectmtxgaussian} amounts to proving that
	\begin{align*}
		V(v) \eqdef \left|\som{|\bar{a}_r[1]|^2|\transp{\bar{a}_r} v|^2}-\Ppa{1+2v[1]^2}\right| \leq {\vrho}
	\end{align*}
	for all $v\in\bbS^{d-1}$. The rest of the proof shows this claim.
	
	Let $\wtilde{a}_r=\Ppa{\bar{a}_r[2],\ldots,\bar{a}_r[d]}$ and $\wtilde{v}=\Ppa{v[2],\ldots,v[d]}.$ We rewrite 
	\[
	|\transp{\bar{a}_r}v|^2=\Ppa{\bar{a}_r[1]v[1]+\transp{\wtilde{a}}_r\wtilde{v}}^2=\Ppa{\bar{a}_r[1]v[1]}^2+\Ppa{\wtilde{a}_r^\top\wtilde{v}}^2+2\bar{a}_r[1]v[1]\wtilde{a}_r^\top\wtilde{v} .
	\]
	We plug this decomposition into $V(v)$ to get 
	\begin{align*}
		V(v)&\leq \left|\som{\bar{a}_r[1]^4-3}\right|v[1]^2 +\left|\som{\bar{a}_r[1]^2-1}\right|\normm{\wtilde{v}}^2+2\left|\som{|\bar{a}_r[1]|^3v[1]\wtilde{a}_r^\top\wtilde{v}}\right| \\
		&+\left|\som{\bar{a}_r[1]^2\Ppa{\wtilde{a}_r^\top\wtilde{v}-\normm{\tilde{v}}^2}}\right|. 
	\end{align*}
	If $X \sim \mathcal{N}(0,1)$ we have $\esp{X^{2p}}=\frac{(2p)!}{2^pp!}$ for $p \in \N$, and in particular $\esp{X^2}=1$ and $\esp{X^4}=3$. By the Tchebyshev's inequality and a union bound argument, $\forall \eps>0,$ and a constant $C(\eps)\approx \max\Ppa{26,\frac{96}{\eps^2}}$ such that when  $m\geq C(\eps)$ we have, 
	\begin{align*}
		\som{\Ppa{\bar{a}_r[1]^4-3}}<\eps,\quad\som{\Ppa{\bar{a}_r[1]^2-1}}<\eps,\quad \som{\bar{a}_r[1]^6\leq20} \\
		\qandq \max\limits_{1\leq r\leq m}|\bar{a}_r[1]|\leq \sqrt{10\log{m}} .
	\end{align*}
	Each of these events happens with probability at least $1-\frac{1}{m^2}$, and thus their intersection occurs with probability at least $1-\frac{4}{m^2}$. On this intersection event, we have 
	\begin{align*}
		V(v)\leq\eps+2\left|\som{\bar{a}_r[1]^3v[1]\wtilde{a}_r^\top\wtilde{v}}\right|
		+\left|\som{\bar{a}_r[1]^2\Ppa{\wtilde{a}_r^\top\wtilde{v}-\normm{\tilde{v}}^2}}\right|.
	\end{align*}
	On the one hand, by Proposition~\ref{pro:vec-hoeffding}, we have 
	\begin{align*}
		\forall \vrho'>0,\quad \left|\som{\bar{a}_r[1]^3v[1]\wtilde{a}_r^\top\wtilde{v}}\right|<\vrho'|v[1]|\normm{\wtilde{v}}^2, 
	\end{align*}
	with a probability  
	\[
	p\geq1-e\exp\Ppa{-\frac{c\vrho'^2m^2}{d\sum_{r=1}^m a_r[1]^6}}\geq 1-e\exp\Ppa{-\frac{c\vrho'^2m}{20d}}\geq 1-\exp\Ppa{-\frac{2Cm}{d}},  
	\]
	where $C$ is a constant that is large enough. When $m\geq \frac{1}{C}d\log(m)$ we get the  bound with probability $p\geq1-\frac{1}{m^2}$.
	On the other hand, by Proposition~\ref{pro:vec-bernstein}, we have 
	\begin{align*}
		\forall \bar{\vrho}>0,\quad \left|\som{\bar{a}_r[1]^2\Ppa{\wtilde{a}_r^\top\wtilde{v}-\normm{\wtilde{v}}^2}}\right|\leq\bar{\vrho}\normm{\wtilde{v}}^2, 
	\end{align*}
	with probability
	\begin{align*}
		p'&\geq 1-\exp\Bba{-\min\Ppa{\frac{\bar{\vrho}^2m^2}{d\sum_{r=1}^m a_r[1]^4},\frac{\bar{\vrho}m}{d\max\limits_{1\leq r\leq m}a_r[1]^2}}},\\
		&\geq 1-\exp\Bba{-\min\Ppa{\frac{\bar{\vrho}^2m}{d(\eps+3)},\frac{\bar{\vrho}m}{10d\log(m)}}},
	\end{align*}
	For $\bar{\vrho}<\frac{\eps+3}{10\log(m)}$, we get  that $p'\geq 1-\exp\Ppa{-\frac{\bar{\vrho}^2m}{d(\eps+3)}}\geq 1-\exp\Ppa{-\frac{2C'm}{d}}$ for $C'$ large enough. Thus, taking again $m\geq \frac{1}{C'}d\log(m)$ we get the  bound with probability $p'\geq1-\frac{1}{m^2}$. Overall, for any $v\in \bbS^{n-1}\cap \T$, we have with probability at least $1-\frac{6}{m^2}$
	\[
	V(v)\leq \eps+\vrho'+2\bar{\vrho}. 
	\]
	We conclude with a covering type argument which can be plugged into the sublinear term and we choose $m\geq C (\vrho,\bar{\vrho})d\log(m)$ and observe that $\log(m)\geq\log(d)$. Therefore, choosing $\vrho=\eps+\vrho'+2\bar{\vrho}$, we get the claim.   
\end{proof}

\end{appendices}

\begin{acknowledgments}
The authors thank the French National Research Agency (ANR) for funding the project FIRST (ANR-19-CE42-0009). 
\end{acknowledgments}
\footnotesize
\bibliographystyle{plain}
{
\bibliography{biblio}
}
\end{document}